\pdfoutput=1
\RequirePackage{ifpdf}
\ifpdf 
\documentclass[pdftex]{sigma}
\else
\documentclass{sigma}
\fi

\numberwithin{equation}{section}

\newtheorem{Theorem}{Theorem}[section]
\newtheorem*{Theorem*}{Theorem}

 { \theoremstyle{definition}

\newtheorem{Remark}[Theorem]{Remark} }

\DeclareMathOperator{\lambdabar}{\kern.2em{\bar{}}\kern-.275em\lambda}
\newcommand{\D}{{\rm d}}
\DeclareMathOperator{\I}{Im}
\DeclareMathOperator{\R}{Re}
\DeclareMathOperator{\Li}{Li}

\DeclareMathOperator{\Span}{span}

\usepackage{esint}

\begin{document}

\allowdisplaybreaks

\renewcommand{\thefootnote}{}

\newcommand{\arXivNumber}{2306.04638}

\renewcommand{\PaperNumber}{074}

\FirstPageHeading

\ShortArticleName{Sun's Series via Cyclotomic Multiple Zeta Values}

\ArticleName{Sun's Series via Cyclotomic Multiple Zeta Values\footnote{This paper is a~contribution to the Special Issue on Asymptotics and Applications of Special Functions in Memory of Richard Paris. The~full collection is available at \href{https://www.emis.de/journals/SIGMA/Paris.html}{https://www.emis.de/journals/SIGMA/Paris.html}}}

\Author{Yajun ZHOU~$^{\rm ab}$}

\AuthorNameForHeading{Y.~Zhou}

\Address{$^{\rm a)}$~Program in Applied and Computational Mathematics (PACM), Princeton University,\\
\hphantom{$^{\rm a)}$}~Princeton, NJ 08544, USA}
\EmailD{\href{yajunz@math.princeton.edu}{yajunz@math.princeton.edu}}
\Address{$^{\rm b)}$~Academy of Advanced Interdisciplinary Studies (AAIS), Peking University,\\
\hphantom{$^{\rm b)}$}~Beijing 100871, P.R.~China}
\EmailD{\href{mailto:yajun.zhou.1982@pku.edu.cn}{yajun.zhou.1982@pku.edu.cn}}

\ArticleDates{Received June 13, 2023, in final form September 29, 2023; Published online October 12, 2023}

\Abstract{We prove and generalize several recent conjectures of Z.-W.~Sun surrounding binomial coefficients and harmonic numbers. We show that Sun's series and their analogs can be represented as cyclotomic multiple zeta values of levels $N\in\{4,8,12,16,24\} $, namely Goncharov's multiple polylogarithms evaluated at $N $-th roots of unity.}


\Classification{11M32; 11B65}

\renewcommand{\thefootnote}{\arabic{footnote}}
\setcounter{footnote}{0}

\section{Introduction}

Amongst Zhi-Wei\ Sun's recent conjectures \cite[Section~2]{Sun2022} concerning binomial coefficients $ \binom mk:= \frac{m!}{k!(m-k)!}$ for $m\in\mathbb Z_{\geq0}$, $k\in\mathbb Z\cap[0,m]$ and harmonic numbers \smash{$ \mathsf H_m^{(r)}:= \sum_{k=1}^m\frac1{k^{r}}$} of order $ r\in\mathbb Z_{>0}$ for $ m\in\mathbb Z_{\geq0}$, one can find certain (${\mathbb Q}$-linear combinations of) convergent series in the following type:\begin{align}
\sum_{k=1}^\infty\frac{\binom{2k}k^n}{k^s}\left( \frac{t}{2^{2n}} \right)^{k}\left[ \prod_{j=1}^M\mathsf H_{k-\smash[t]{1}}^{(r_j)} \right]\left[\prod_{\vphantom{j}\smash[b]{j'}=1}^{M'}{{\mathsf H}}_{\smash[t]{2}k-1}^{( {r}'_{j'})}\right],\label{eq:binomEuler}
\end{align}
where $n\in\{-1,1\}$, $s\in\mathbb Z$, $ M,M'\in\mathbb Z_{\geq0}$, and $ |t|<1$.
Such infinite series were previously known as ``inverse binomial sums'' and ``binomial sums'' in some literature of high energy physics
\cite{Ablinger2017,ABRS2014,DavydychevKalmykov2001,DavydychevKalmykov2004, KalmykovVeretin2000,Kalmykov2007,
Weinzierl2004bn} that revolved around a certain class of Feynman diagrams. If one allows $n=0$ in \eqref{eq:binomEuler}, then one recovers some special cases of the ``nested harmonic sums'' studied by mathematical physicists \cite{ABS2011,ABS2013}.

In our recent work \cite[Section~3]{Zhou2022mkMpl}, we have analyzed \eqref{eq:binomEuler} for $n\in\{-1,0,1\}$, effectively reaching the following conclusion:
if $t$ is an algebraic number satisfying $ |t|<1$, then the infinite sum in \eqref{eq:binomEuler} can be represented as a $ \overline{\mathbb Q}(\pi)$-linear combination of Goncharov's multiple polylogarithms (MPLs) \cite{Goncharov1997,Goncharov1998}
\begin{align}
\Li_{a_1,\dots,a_n}(z_1,\dots,z_n):= \sum_{\ell_{1}>\dots>\ell_{n}>0}\prod_{j=1}^n\frac{z_{j}^{\ell_{j}}}{\ell_j^{a_j}},
\label{eq:Mpl_defn}\end{align}where $ a_1,\dots,a_n\in\mathbb Z_{>0}$ are positive integers and $ z_1,\dots,z_n\in\overline{\mathbb Q}$ are explicitly computable algebraic numbers. For meticulously chosen algebraic numbers $ t\in\overline{\mathbb Q}$, we can sometimes confine the arguments $ z_1,\dots,z_n$ of our MPLs to roots of unity, arriving at members of \begin{gather}
\mathfrak Z_{ k}(N):= \Span_{\mathbb Q}\biggl\{\Li_{a_1,\dots,a_n}(z_1,\dots,z_n)\in\mathbb C\Bigm| \nonumber\\
\hphantom{\mathfrak Z_{ k}(N):= \Span_{\mathbb Q}\biggl\{}{}
a_1,\dots,a_n\in\mathbb Z_{>0},\, z_{1}^{N}=\dots=z_n^N=1,\, \sum _{j=1}^{n}a_{j}=k  \biggr\},\label{eq:Zk(N)_defn}
\end{gather}the $ \mathbb Q$-vector space spanned by cyclotomic multiple zeta values (CMZVs) of weight $k\in\mathbb Z_{>0}$ and level $N\in\mathbb Z_{>0}$.
In \cite[Section~3]{Zhou2022mkMpl}, we have achieved our aforementioned goals by converting \eqref{eq:binomEuler} into certain $ \overline{\mathbb Q}(\pi)$-linear combinations of contour integrals \begin{align}
\int_C \Li_{a_1,\dots,a_n}(z_1(t),\dots,z_n(t))\frac{\D t}{R(t)}
\label{eq:MPL_int_repn}\end{align}for $ z_1(t),\dots,z_n(t),R(t)\in\overline{\mathbb Q}(t)$, and evaluating these integrals via fibrations of MPLs \cite[Lem\-ma~2.14 and Corollary~3.2]{Panzer2015}. Originating from the seminal works of Davydychev--Kalmykov~\cite{DavydychevKalmykov2001,DavydychevKalmykov2004}, the integral formulation \eqref{eq:MPL_int_repn} of the infinite series in \eqref{eq:binomEuler} was followed up by Weinzierl~\cite{Weinzierl2004bn}, Ablinger~\cite{Ablinger2017,Ablinger2022}, Au \cite{Au2020,Au2022a}, and Xu--Zhao \cite{XuZhao2022a,XuZhao2022b} among other researchers, prior to our study in~\cite[Section~3]{Zhou2022mkMpl}.

Special values of MPLs at algebraic arguments already cover a wide class of mathematical constants, such as $ \log2=\Li_1\big(\frac12\big)=-\Li_1(-1)$, $G=\frac{1}{2{\rm i}}\left[\Li_2(\rm i)-\Li_2(-\rm i)\right]$, and $ \zeta(3)=\Li_3(1)$. Furthermore, the versatility of MPLs allows us to compute sophisticated series whose exact values are hard to discover empirically. For instance, from \cite[Table 7]{Zhou2022mkMpl} one may read off\footnote{In what follows, we abbreviate $ \mathsf H_m^{(1)}$ into $ \mathsf H_m^{}$.}\begin{gather*}\sum_{k=0}^\infty\frac{\binom{2k}k}{(2k+1)^{3}}\frac{(-1)^k}{2^{4k}}={}\frac{25\R\Li_3\big({\rm e}^{2\pi {\rm i}/5}\big)}{12}+\frac{\zeta(3)}{2}\in\mathfrak Z_3(5),\\\sum_{k=0}^\infty\frac{\binom{2k}k\mathsf H_{2k+1}}{(2k+1)^{2}}\frac{(-1)^k}{2^{4k}}={}-5\R\Li_3\big({\rm e}^{2\pi {\rm i}/5}\big)+\frac{8\zeta(3)}{5}\in\mathfrak Z_3(5),\end{gather*}which combine into an identity \begin{align*}
\sum_{k=0}^\infty\frac{\binom{2k}{k}\big( 5\mathsf H_{2k+1} +\frac{12}{2k+1}\big)}{(2k+1)^2(-16)^k}=14\zeta(3)
\end{align*}found by Sun and proved by Charlton--Gangl--Lai--Xu--Zhao \cite[Section~4]{Charlton2022SunConj}.

In Section~\ref{sec:centr_binom} of this note, we extrapolate \cite[Corollary 3.8]{Zhou2022mkMpl} to the theorem below.
\begin{Theorem}[Sun's series involving $ \binom{2k}k$ and harmonic numbers]\label{thm:2k}
For $ r\in\mathbb Z_{>1}$, we have the following CMZV characterizations:
\begin{gather}
  \sum_{k=0}^\infty\frac{\binom{2k}k}{8^k}\mathsf H_k^{(r)}=  \ointctrclockwise_{|z|=1}\frac{\Li_r\left( \frac{(1+z)^{2}}{8z} \right)}{1-\frac{(1+z)^{2}}{8z}}\frac{\D z}{2\pi {\rm i}z}\in\sqrt{2}\mathfrak Z_{r}(8),\label{eq:2kZ8}\\
 \sum_{k=0}^\infty\frac{\binom{2k}k}{8^k}\mathsf H_{2k}^{(r)}= \ointctrclockwise_{|z|=1}\left[\frac{\Li_r\left( \frac{1+z^{2}}{2\sqrt{2}z} \right)}{1- \frac{1+z^{2}}{2\sqrt{2}z}}+\frac{\Li_r\left( -\frac{1+z^{2}}{2\sqrt{2}z} \right)}{1+ \frac{1+z^{2}}{2\sqrt{2}z}}\right]\frac{\D z}{4\pi {\rm i}z}\in\sqrt{2}\mathfrak Z_{r}(8),\label{eq:2kZ8'}\\
 \sum_{k=0}^\infty\frac{\binom{2k}k}{16^k}\mathsf H_k^{(r)}=  \ointctrclockwise_{|z|=1}\frac{\Li_r\left( \frac{(1+z)^{2}}{16z} \right)}{1-\frac{(1+z)^{2}}{16z}}\frac{\D z}{2\pi {\rm i}z}\in\sqrt{3}\mathfrak Z_{r}(12),\label{eq:2kZ12}\\
  \sum_{k=0}^\infty\frac{\binom{2k}k}{16^k}\mathsf H_{2k}^{(r)}= \ointctrclockwise_{|z|=1}\left[\frac{\Li_r\left( \frac{1+z^{2}}{4z} \right)}{1- \frac{1+z^{2}}{4z}}+\frac{\Li_r\Bigl( -\frac{1+z^{2}}{4z} \Bigr)}{1+ \frac{1+z^{2}}{4z}}\right]\frac{\D z}{4\pi {\rm i}z}\in\sqrt{3}\mathfrak Z_{r}(12).\label{eq:2kZ12'}
\end{gather} \end{Theorem}

\begin{Remark}\label{tab:lin_binom_Sun}
 For $ r=2$ (resp.\ $ r=3$), one may reorganize the formulas below into series involving $\mathsf H_{2k}^{(r)}-\frac1{2^{r}}\mathsf H_k^{(r)} $ in \cite[Theorem 1.1]{Sun2022} (resp.\ \cite[Conjectures~2.2 and~2.3]{Sun2022}):
\begin{gather*}
\sum_{k=0}^\infty\frac{\binom{2k}k}{8^k}\mathsf H_k^{(2)}=   -12\sqrt{2}\Li_2\big(\sqrt{2}-1\big)-4\sqrt{2}\widetilde{\lambda}^2+\frac{11 \pi ^2}{6 \sqrt{2}}\in\sqrt{2}\mathfrak Z_{2}(8),\\
 \sum_{k=0}^\infty\frac{\binom{2k}k}{8^k}\mathsf H_k^{(3)}=  -\frac{128\sqrt{2}}{3}\R \Li_3\big({\rm e}^{\pi {\rm i}/4}\big)+72\sqrt{2}\Li_3\biggl( \frac{1}{\sqrt{2}} \biggr)-48\sqrt{2}\Li_3\big(\sqrt{2}-1\big)\\
\hphantom{\sum_{k=0}^\infty\frac{\binom{2k}k}{8^k}\mathsf H_k^{(2)}=}{}
-\frac{39\zeta(3)}{4\sqrt{2}} +12\sqrt{2}\big(3\lambda-2\widetilde \lambda\big)\Li_2\big(\sqrt{2}-1\big)\\
\hphantom{\sum_{k=0}^\infty\frac{\binom{2k}k}{8^k}\mathsf H_k^{(2)}=}{}
+\frac{9 \big(\pi ^{2}-\lambda^{2}\big ) \lambda- \big[12 \widetilde{\lambda } \big(2 \widetilde{\lambda }-9 \lambda \big)+54 \lambda ^2+11 \pi ^2\big]\widetilde{\lambda }}{3 \sqrt{2}}\in\sqrt{2}\mathfrak Z_{3}(8),\\
\sum_{k=0}^\infty\frac{\binom{2k}k}{8^k}\mathsf H_{2k}^{(2)}= \frac{\pi^{2}}{16\sqrt{2}}+\frac{1}{4}\sum_{k=0}^\infty\frac{\binom{2k}k}{8^k}\mathsf H_k^{(2)}\in\sqrt{2}\mathfrak Z_{2}(8),\\
\sum_{k=0}^\infty\frac{\binom{2k}k}{8^k}\mathsf H_{2k}^{(3)}= \frac{35\sqrt2\zeta(3)}{64}-\frac{\sqrt2\pi G}8+\frac{1}{8}\sum_{k=0}^\infty\frac{\binom{2k}k}{8^k}\mathsf H_k^{(3)}\in\sqrt{2}\mathfrak Z_{3}(8),\\
\sum_{k=0}^\infty\frac{\binom{2k}k}{16^k}\mathsf H_k^{(2)}= -\frac{32\Li_2\big(2-\sqrt{3}\big)}{\sqrt{3}}-\frac{4\widetilde{\varLambda}^2}{\sqrt{3}}+\frac{5 \pi ^2}{3 \sqrt{3}}\in\sqrt{3}\mathfrak Z_{2}(12),\\
\sum_{k=0}^\infty\frac{\binom{2k}k}{16^k}\mathsf H_{k}^{(3)}= -\frac{288}{\sqrt{3}}\R \Li_3\big({\rm e}^{\pi {\rm i}/6}\big)-\frac{256\Li_3\big(2-\sqrt{3}\big)}{\sqrt{3}}-\frac{64}{\sqrt{3}}\Li_3\biggl(1-\frac{\sqrt{3}}{2}\biggr)\\
\hphantom{\sum_{k=0}^\infty\frac{\binom{2k}k}{16^k}\mathsf H_{k}^{(3)}=}{}
+\frac{256}{\sqrt{3}}\Li_3\biggl(\frac{\sqrt{3}-1}{2}\biggr)+76\sqrt{3}\zeta(3)+\frac{64\big(2\lambda-\widetilde \varLambda\big)\Li_2\big(2-\sqrt{3}\big)}{\sqrt{3}}\\
\hphantom{\sum_{k=0}^\infty\frac{\binom{2k}k}{16^k}\mathsf H_{k}^{(3)}=}{}
+\frac{2 \big[24 \big( \widetilde{\varLambda }- \lambda\big)\lambda\widetilde{\varLambda }-5 \pi^2\widetilde{\varLambda }+8 \lambda ^3\big]}{3 \sqrt{3}}\in\sqrt{3}\mathfrak Z_{3}(12),\\
\sum_{k=0}^\infty\frac{\binom{2k}k}{16^k}\mathsf H_{2k}^{(2)}= \frac{\pi^{2}}{36\sqrt{3}}+\frac{1}{4}\sum_{k=0}^\infty\frac{\binom{2k}k}{16^k}\mathsf H_k^{(2)}\in\sqrt{3}\mathfrak Z_{2}(12)\vphantom{\frac{\frac{\int1}{1}}{1}},\\
\sum_{k=0}^\infty\frac{\binom{2k}k}{16^k}\mathsf H_{2k}^{(3)}= \frac{2\zeta(3)}{3\sqrt3}-\frac{\pi\I\Li_2\big({\rm e}^{2\pi {\rm i}/3}\big)}{4\sqrt{3}}+\frac{1}{8}\sum_{k=0}^\infty\frac{\binom{2k}k}{16^k}\mathsf H_k^{(3)}\in\sqrt{3}\mathfrak Z_{3}(12),
\end{gather*}
where
\[
\lambda:= \log2 , \qquad \widetilde \lambda:= \log\big(1+\sqrt{2}\big) , \qquad \widetilde{\varLambda}:= \log\big(2+\sqrt{3}\big).
\]
\end{Remark}

In Section~\ref{sec:non-centr_binom}, we adapt our recent work \cite[Corollary 3.8]{Zhou2022mkMpl} to scenarios involving the binomial coefficients $ \binom{3k}k$ and $ \binom{4k}{2k}$, thereby accommodating to some of Sun's new conjectures \cite[Section~2]{Sun2022} announced on June 5, 2023. Our results are summarized in the next two theorems.\looseness=1

\pagebreak

\begin{Theorem}[CMZVs in Sun's series involving $ \binom{3k}k$, $ \binom{4k}{2k}$ and harmonic numbers] \label{thm:3k4k}\quad
 \begin{enumerate}\itemsep=0pt
 \item[$(a)$] We have
\begin{align}
 &\sum_{k=1}^\infty\frac{1}{k^{s}2^k\binom{3k}k}=\begin{cases}
\displaystyle \int_0^1\frac{\frac{t(1-t)^2}{2}}{1-\frac{t(1-t)^2}{2}}\frac{\D t}{t} =\frac{\pi}{10}-\frac{\log 2}{5},& s=1, \vspace{2mm}\\
\displaystyle \int_0^1\Li_{s-1}\left( \frac{t(1-t)^2}{2} \right)\frac{\D t}{t}\in\mathfrak Z_{s}(4), & s\in\mathbb Z_{>1} ,
\end{cases}\label{eq:3kZ4}\\
 &\sum_{k=1}^\infty\frac{\mathsf H_{3k}-\mathsf H_k+\frac{1}{k}}{k^{s}2^k\binom{3k}k}
 =\begin{cases}
\displaystyle -\int_0^1\frac{\frac{t(1-t)^2}{2}\log t}{1-\frac{t(1-t)^2}{2}}\frac{\D t}{t} =\frac{2G}{5}-\frac{\log^{2}2}{10},& s=1, \vspace{1mm}\\
\displaystyle -\int_0^1\Li_{s-1}\left( \frac{t(1-t)^2}{2} \right)\log t\frac{\D t}{t}\in\mathfrak Z_{s+1}(4), & s\in\mathbb Z_{>1} ,
\end{cases}\label{eq:3kZ4'}\\
 & \sum_{k=1}^\infty\frac{\mathsf H_{3k}-\mathsf H_{2k}}{k^{s}2^k\binom{3k}k}=
 \begin{cases}\displaystyle-\int_0^1\frac{\frac{t(1-t)^2}{2}\log(1- t)}{1-\frac{t(1-t)^2}{2}}\frac{\D t}{t} &\vspace{1mm} \\
\displaystyle \qquad =\frac{2G}{5}+\frac{\log ^22}{10}-\frac{\pi \log2}{20} -\frac{\pi ^2}{40},& s=1, \vspace{1mm}\\
\displaystyle-\int_0^1\Li_{s-1}\left( \frac{t(1-t)^2}{2} \right)\log(1- t)\frac{\D t}{t}\in\mathfrak Z_{s+1}(4), & s\in\mathbb Z_{>1} .
\end{cases}\!\!\!\!\!\label{eq:3kZ4''}
\end{align}
\item[$(b)$] For $ r\in\mathbb Z_{>0}$, the following formulas hold true:
\begin{gather}
\sum_{k=0}^\infty\frac{2^k\binom{3k}k}{3^{3k}}\mathsf H_k^{(r)}=\frac{3\sqrt{3}}{2\pi}\int_0^\infty\frac{\Li_r\big( \frac{2X^{3}}{(1+X^3)^{2}} \big)}{1-\frac{2X^{3}}{(1+X^3)^{2}}}\frac{\D X}{1+X^3}\in\mathfrak Z_r(12)+\sqrt{3}\mathfrak Z_{r}(12),\label{eq:3kHkZr}\\
\sum_{k=0}^\infty\frac{2^k\binom{3k}k}{3^{3k}}\mathsf H_{2k}^{(r)}=\frac{3\sqrt{3}}{2\pi}\int_0^\infty\left[\frac{\Li_r\left( \frac{\sqrt{2}x^{3}}{1+x^6} \right)}{1-\frac{\sqrt{2}x^{3}}{1+x^6}}+\frac{\Li_r\left(- \frac{\sqrt{2}x^{3}}{1+x^6} \right)}{1+\frac{\sqrt{2}x^{3}}{1+x^6}}\right]\frac{x\D x}{1+x^6}\nonumber\\
\hphantom{\sum_{k=0}^\infty\frac{2^k\binom{3k}k}{3^{3k}}\mathsf H_{2k}^{(r)}=}{}
\in\mathfrak Z_r(24)+\sqrt{3}\mathfrak Z_{r}(24),\label{eq:3kH2kZr}\\
\sum_{k=0}^\infty\frac{2^k\binom{3k}k}{3^{3k}}\big(
3\mathsf H_{3 k}^{}-2\mathsf H_{2 k}^{}-\mathsf H_k^{}-3 \log 3\big)\mathsf H_k^{(r)}\nonumber\\
\qquad=\frac{3\sqrt{3}}{2\pi}\int_0^\infty\frac{\Li_r\big( \frac{2X^{3}}{(1+X^3)^{2}} \big)}{1-\frac{2X^{3}}{(1+X^3)^{2}}}\frac{\log\frac{X^3}{(1+X^3)^{2}}\D X}{1+X^3}\in\mathfrak Z_{r+1}(12)+\sqrt{3}\mathfrak Z_{r+1}(12).\label{eq:3kH3kZr}
\end{gather}
In addition, one has
\begin{gather}
\sum_{k=0}
^\infty\frac{2^k\binom{3k}k}{3^{3k}}\big(
3\mathsf H_{3 k}-2\mathsf H_{2 k}-\mathsf H_k-3 \log 3\big)=\frac{3\sqrt{3}}{2\pi}\int_0^\infty\frac{\log\frac{X^3}{(1+X^3)^{2}}}{1-\frac{2X^{3}}{(1+X^3)^{2}}}\frac{\D X}{1+X^3}\nonumber\\
\qquad=-\frac{3+\sqrt{3}}{2}\log2-\sqrt{3}\log\big(2+\sqrt3\big)\in \mathfrak Z_1(12)+\sqrt{3}\mathfrak Z_{1}(12).\label{eq:3kH3kZ1}
\end{gather}
\item[$(c)$] For $r\in\mathbb Z_{>0}$, we have
\begin{gather}
\sum_{k=0}^\infty\frac{\binom{4k}{2k}}{2^{5k}}\mathsf H_k^{(r)}= \frac{2\sqrt{2}}{\pi}\int_0^\infty\frac{\Li_r\big( \frac{2X^{4}}{(1+X^4)^{2}} \big)}{1-\frac{2X^{4}}{(1+X^4)^{2}}}\frac{X^{2}\D X}{1+X^4}\nonumber\\
\hphantom{\sum_{k=0}^\infty\frac{\binom{4k}{2k}}{2^{5k}}\mathsf H_k^{(r)}=}{}
\in
 \sqrt{1+\frac{1}{\sqrt{2}}}\mathfrak Z_{r}(16)+\sqrt{1-\frac{1}{\sqrt{2}}}\mathfrak Z_{r}(16),\label{eq:4kZr16}\\
\sum_{k=0}^\infty\frac{3^k\binom{4k}{2k}}{2^{6k}}\mathsf H_k^{(r)}= \frac{2\sqrt{2}}{\pi}\int_0^\infty\frac{\Li_r\bigl( \frac{3X^{4}}{(1+X^4)^{2}} \bigr)}{1-\frac{3X^{4}}{(1+X^4)^{2}}}\frac{X^{2}\D X}{1+X^4} \in \mathfrak Z_r(24)+\sqrt{3}\mathfrak Z_{r}(24).
\label{eq:4kZr24}
\end{gather}
\end{enumerate}
\end{Theorem}

\begin{Remark}Part (a) of the theorem above uses ideas from Kam Cheong Au (see \cite[Example~4.21]{Au2020} and \cite[Proposition 7.13]{Au2022a}). Some formulas below
were previously reported by Au, while the rest 
prove and generalize Sun's experimental identities \cite[formulas~(2.6) and~(2.7) in Conjecture~2.4]{Sun2022}:
\begin{gather*}
\sum_{k=1}^\infty\frac{1}{k^{2}2^k\binom{3k}k}=  \frac{\pi ^2}{24}-\frac{\lambda^{2}}{2}\in\mathfrak Z_2(2)\subset\mathfrak Z_2(4),\\
\sum_{k=1}^\infty\frac{1}{k^{3}2^k\binom{3k}k}=  -\frac{33 \zeta (3)}{16}+\pi G+\frac{\lambda ^3}{6}-\frac{\pi ^2 \lambda}{24}\in\mathfrak Z_3(4),\\
\sum_{k=1}^\infty\frac{1}{k^{4}2^k\binom{3k}k}=  -\frac{21}{2} \Li_4\left(\frac{1}{2}\right)+2 \pi \I\Li_3\left(\frac{1+{\rm i}}{2}\right)\\
\hphantom{\sum_{k=1}^\infty\frac{1}{k^{4}2^k\binom{3k}k}=}{}
-\frac{57\zeta (3) \lambda}{8} -\frac{23 \lambda ^4}{48} +\frac{19\pi ^2\lambda^2}{48} +\frac{61 \pi ^4}{960}\in\mathfrak Z_4(4),\\
\sum_{k=1}^\infty\frac{\mathsf H_{3k}-\mathsf H_k}{k2^k\binom{3k}k}= \frac{2G}{5}+\frac{2\lambda^{2}}{5}-\frac{\pi ^2}{24}\in\mathfrak Z_2(2)+{\rm i}\mathfrak Z_2(4),\\
 \sum_{k=1}^\infty\frac{\mathsf H_{3k}-\mathsf H_k}{k^{2}2^k\binom{3k}k}= \frac{11 \zeta (3)}{4}-\pi G-\frac{\pi ^2 \lambda}{24} \in\mathfrak Z_3(4),\\ \sum_{k=1}^\infty\frac{\mathsf H_{3k}-\mathsf H_k}{k^{3}2^k\binom{3k}k}= \frac{27}{2} \Li_4\left(\frac{1}{2}\right)-2 \pi\I\Li_3\left(\frac{1+{\rm i}}{2}\right)\\
 \hphantom{\sum_{k=1}^\infty\frac{\mathsf H_{3k}-\mathsf H_k}{k^{3}2^k\binom{3k}k}=}{}
 +\frac{145\zeta (3) \lambda}{16} +2G^2+\frac{9 \lambda^{4}}{16}-\frac{23\pi ^2 \lambda^{2}}{48}-\frac{17 \pi ^4}{160}\in\mathfrak Z_4(4), \\
\sum_{k=1}^\infty\frac{\mathsf H_{2k}-\mathsf H_k}{k2^k\binom{3k}k}= \frac{3\lambda^{2}}{10}+\frac{\pi \lambda}{20} -\frac{\pi ^2}{60}\in\mathfrak Z_2(2)+{\rm i}\mathfrak Z_2(4),\\
\sum_{k=1}^\infty\frac{\mathsf H_{2k}-\mathsf H_k}{k^{2}2^k\binom{3k}k}= \frac{33 \zeta (3)}{32}-\frac{\pi G}{2}+\frac{\pi ^2 \lambda}{24} \in\mathfrak Z_3(4),\\
\sum_{k=1}^\infty\frac{\mathsf H_{2k}-\mathsf H_k}{k^{3}2^k\binom{3k}k}= \frac{9}{2} \Li_4\left(\frac{1}{2}\right)+\frac{93\zeta (3) \lambda}{32} +\frac{\pi G\lambda}{2}+\frac{3 \lambda^{4}}{16}-\frac{5\pi ^2 \lambda^{2}}{24} -\frac{31 \pi ^4}{640}\in\mathfrak Z_4(4),
\end{gather*}
where $\lambda:= \log2 $ and $ G:= \I\Li_2({\rm i})$.

It is possible to extend the statements in part (a) to $ s\in\mathbb Z_{\leq0}$, but the resulting patterns are not as neat as the $ s\in\mathbb Z_{>1}$ cases. For example, we have
 \begin{align*}\begin{split}
\sum_{k=1}^\infty\frac{\mathsf H_{3k}-8\mathsf H_{2k}+7\mathsf H_k}{2^k\binom{3k}k}={}&\int_0^1  x\frac{\D }{\D x}\frac{x}{1-x}\bigg|_{x=\frac{t(1-t)^2}{2}}\frac{\log\frac{t^7}{(1-t)^8}}{t}\D t+7\int_0^1\frac{\frac{t(1-t)^2}{2}}{1-\frac{t(1-t)^2}{2}}\frac{\D t}{t}\\={}&\frac{22G}{125}-\frac{2 \log ^22}{25}-\frac{22\pi \log 2}{125}-\frac{17 \pi ^2}{1500}-\frac{33 \log 2}{25}+\frac{11 \pi }{50} \end{split}
\end{align*}
and
\begin{align*}\sum_{k=1}^\infty\frac{\mathsf H_{3k}-8\mathsf H_{2k}+7\mathsf H_k}{2^k\binom{3k}k}k={}&\int_0^1  \left(x\frac{\D }{\D x}\right)^2\frac{x}{1-x}\bigg|_{x=\frac{t(1-t)^2}{2}}\frac{\log\frac{t^7}{(1-t)^8}}{t}\D t\\
&{}+7\int_0^1 x\frac{\D }{\D x}\frac{x}{1-x}\bigg|_{x=\frac{t(1-t)^2}{2}}\frac{\D t}{t}\\={}&\frac{316G}{3125}-\frac{6 \log ^22}{625}-\frac{316 \pi \log 2}{3125}-\frac{17 \pi ^2}{12500}-\frac{549 \log 2}{625}+\frac{33 \pi }{1250} .
\end{align*}
These two equations conjoin into a succinct formula\begin{align*}
\sum_{k=1}^\infty\frac{\mathsf H_{3k}-8\mathsf H_{2k}+7\mathsf H_k}{2^k\binom{3k}k}(25k-3)=2G-2 (\pi+9 ) \log 2
\end{align*} that has been discovered recently by Sun \cite[formula (2.10) in Conjecture 2.5]{Sun2022}. \end{Remark}

\begin{Remark}\label{tab:3kSun'}
Part (b) of the theorem above and its accompanying formulas below 
verify and extend Sun's recent observations \cite[formulas (2.12)--(2.13) in Conjecture 2.6]{Sun2022}:
\begin{gather*}
\sum_{k=0}
^\infty\frac{2^k\binom{3k}k}{3^{3k}}\mathsf H_k= \dfrac{3-\sqrt{3}}{2} \lambda-\dfrac{3\big(\sqrt{3}+1\big)}4\varLambda+\sqrt{3}\widetilde{\varLambda}\in\mathfrak Z_1(12)+\sqrt{3}\mathfrak Z_{1}(12)\vphantom{\frac{\frac{\frac\int1}{1}}{1}},\\
\sum_{k=0}
^\infty\frac{2^k\binom{3k}k}{3^{3k}}\mathsf H_k^{(2)}= -12 \sqrt{3} \Li_2\bigg(\dfrac{\sqrt{3}-1}{2}\bigg)+3\big( \sqrt{3}-1\big) \Li_2\big(2-\sqrt{3}\big)+\dfrac{3\big( \sqrt{3}-1\big)}{2} \Li_2\left(\dfrac{1}{4}\right)\\
\hphantom{\sum_{k=0}^\infty\frac{2^k\binom{3k}k}{3^{3k}}\mathsf H_k^{(2)}=}{}
+\dfrac{3\big[\big(2 \sqrt{3}-3\big)\lambda^2-2\big(2\sqrt{3}-1\big)\lambda \widetilde{\varLambda }- \widetilde{\varLambda }^{2}\big]}{4}+\dfrac{\big(7 \sqrt{3}+3\big)\pi ^2}{12}\\
\hphantom{\sum_{k=0}^\infty\frac{2^k\binom{3k}k}{3^{3k}}\mathsf H_k^{(2)}=}{}
 \in\mathfrak Z_2(12)+\sqrt{3}\mathfrak Z_{2}(12),
\\
\sum_{k=0}
^\infty\frac{2^k\binom{3k}k}{3^{3k}}\mathsf H_{2k}= \dfrac{3-\sqrt{3}}{4} \lambda-\dfrac{3\big(\sqrt{3}+1)}4\big(\varLambda-\widetilde{\varLambda}\big)\in\mathfrak Z_1(12)+\sqrt{3}\mathfrak Z_{1}(12)\\
\hphantom{\sum_{k=0}^\infty\frac{2^k\binom{3k}k}{3^{3k}}\mathsf H_{2k}=}{}
\subset \mathfrak Z_1(24)+\sqrt{3}\mathfrak Z_{1}(24),\\
\sum_{k=0}
^\infty\frac{2^k\binom{3k}k}{3^{3k}}\hbar_k= -\dfrac{3+\sqrt{3}}{2}\lambda-\sqrt{3}\widetilde{\varLambda}\in\mathfrak Z_1(12)+\sqrt{3}\mathfrak Z_{1}(12)\vphantom{\frac{\frac{\frac\int1}{1}}{1}},\\
\sum_{k=0}
^\infty\frac{2^k\binom{3k}k}{3^{3k}}\hbar_k\mathsf H_k= -30 \sqrt{3} \Li_2\!\bigg(\dfrac{\sqrt{3}-1}{2}\bigg)+3\big( \sqrt{3}-1\big) \Li_2\big(2-\sqrt{3}\big)+\dfrac{15}{4}\big( \sqrt{3}-1\big) \Li_2\!\left(\dfrac{1}{4}\right)\\
\hphantom{\sum_{k=0}^\infty\frac{2^k\binom{3k}k}{3^{3k}}\hbar_k\mathsf H_k=}{}
+\dfrac{6\big(1-5 \sqrt{3}\big) \lambda \widetilde{\varLambda }+6 \sqrt{3} \varLambda \widetilde{\varLambda }-3\big(3 \sqrt{3}+1\big) \widetilde{\varLambda }^2+\big(17 \sqrt{3}-33\big) \lambda ^2}{4}\\
\hphantom{\sum_{k=0}^\infty\frac{2^k\binom{3k}k}{3^{3k}}\hbar_k\mathsf H_k=}{}
+\frac{3\big( \sqrt{3}+3\big) \lambda \varLambda}{4}+\dfrac{\big(13 \sqrt{3}+3\big)\pi ^2}{8} \in\mathfrak Z_2(12)+\sqrt{3}\mathfrak Z_{2}(12),
\end{gather*}
where $\lambda:= \log2 $, $ \varLambda:= \log3$, $ \widetilde{\varLambda}:= \log\big(2+\sqrt{3}\big)$, and $ \hbar_k:= 3\mathsf H_{3 k}-2\mathsf H_{2 k}-\mathsf H_k-3 \log 3$.
\end{Remark}

\begin{Remark}In part (c) of the theorem above, we do not need dedicated spaces for convergent series like (cf.\ \cite[Remark 2.8]{Sun2022})\begin{align}
\sum_{k=0}^\infty{\binom{4k}{2k}}z^{2k}\mathsf H_{2k}^{(r)}={}&\frac{1}{2}
\sum_{k=0}^\infty{\binom{2k}{k}}\big[z^{k}+(-z)^k\big]\mathsf H_{k}^{(r)},\label{eq:4k2k_a}\\
\sum_{k=0}^\infty{\binom{4k}{2k}}z^{2k}\mathsf H_{4k}^{(r)}={}&\frac{1}{2}
\sum_{k=0}^\infty{\binom{2k}{k}}\big[z^{k}+(-z)^k\big]\mathsf H_{2k}^{(r)},\label{eq:4k2k_b}
\end{align}since the right-hand sides of these two equations are fully characterized by \cite[formulas (1.1) and~(1.2)]{Sun2022} for $ r=1$, and by \eqref{eq:chi_fib''}--\eqref{eq:ups_fib''} in Section~\ref{subsec:CMZV_8_12} below for $ r\in\mathbb Z_{>1}$.\end{Remark}

\begin{Theorem}[logarithms in Sun's series involving $ \binom{3k}k$, $ \binom{4k}{2k}$ and harmonic numbers]\label{thm:logSun}\quad
\begin{enumerate}\itemsep=0pt
\item[$(a)$] If $\varphi\in(0,\pi/3)$, then we have
\begin{gather}
\sum_{k=0}^\infty\frac{\binom{3k}k}{3^{3k}}\left( 4\cos^2\frac{3\varphi}{2} \right)^k\mathsf H_k=\frac{3\sqrt{3}}{2\pi}\int_0^\infty\frac{\Li_1\Big( \frac{4X^{3}\cos^2\frac{3\varphi}{2}}{(1+X^3)^{2}} \Big)}{1-\frac{4X^{3}\cos^2\frac{3\varphi}{2}}{(1+X^3)^{2}}}\frac{\D X}{1+X^3}\nonumber\\
\qquad=\frac{\sqrt{3}}{\sin\frac{3\varphi}{2}}\left[ \cos \frac{\varphi }{2} \log \frac{2 \cos ^2\left(\frac{\varphi }{2}-\frac{\pi }{3}\right)}{\sqrt{3} \sin \varphi } +\cos \left(\frac{\varphi }{2}-\frac{\pi }{3}\right) \log \frac{2 \cos ^2\frac{\varphi }{2}}{\sqrt{3 }\sin \big(\frac{2 \pi }{3}-\varphi \big)}\right],\label{eq:3kHklog}\\
\sum_{k=0}^\infty\frac{\binom{3k}k}{3^{3k}}\left( 4\cos^2\frac{3\varphi}{2} \right)^k(
3\mathsf H_{3 k}-2\mathsf H_{2 k}-\mathsf H_k-3 \log 3)\nonumber\\
\qquad=\frac{3\sqrt{3}}{2\pi}\int_0^\infty\frac{\log\frac{X^{3}}{(1+X^3)^{2}}}{1-\frac{4X^{3}\cos^2\frac{3\varphi}{2}}{(1+X^3)^{2}}}\frac{\D X}{1+X^3}\nonumber\\
\qquad=-\frac{2\sqrt{3}}{\sin\frac{3\varphi}{2}}\left[ \cos\frac{\varphi }{2} \log \left(2\cos \left(\frac{\varphi }{2}-\frac{\pi }{3}\right)\right)+\cos \left(\frac{\varphi }{2}-\frac{\pi }{3}\right) \log \left(2 \cos \frac{\varphi }{2}\right) \right],\label{eq:3kH3klog}\\
\sum_{k=0}^\infty\frac{\binom{3k}k}{3^{3k}}\left( 2\cos\frac{3\varphi}{2} \right)^{2k}\mathsf H_{2k}\nonumber\\
\qquad=\frac{3\sqrt{3}}{2\pi}\int_0^\infty\left[\frac{\Li_1\Big( \frac{2x^{3}\cos\frac{3\varphi}{2}}{1+x^6} \Big)}{1-\frac{2x^{3}\cos\frac{3\varphi}{2}}{1+x^6}}+\frac{\Li_1\Big(\!-\! \frac{2x^{3}\cos\frac{3\varphi}{2}}{1+x^6} \Big)}{1+\frac{2x^{3}\cos\frac{3\varphi}{2}}{1+x^6}}\right]\frac{x\D x}{1+x^6}=\frac{\sqrt{3}}{\sin\frac{3\varphi}{2}}\nonumber\\
\qquad\times\left[ \cos \frac{\varphi }{2} \log \frac{2 \cos \frac{\varphi }{2}\cos \left(\frac{\varphi }{2}-\frac{\pi }{3}\right)}{\sqrt{3 }\sin \varphi } +\cos \left(\frac{\varphi }{2}-\frac{\pi }{3}\right) \log \frac{2\cos \frac{\varphi }{2}\cos \left(\frac{\varphi }{2}-\frac{\pi }{3}\right)}{\sqrt{3} \sin\big(\frac{2 \pi }{3}-\varphi \big)} \right].\!\!\!\label{eq:3kH2klog}
\end{gather}
\item[$(b)$] If $\psi \in(0,\pi/4)$, then we have
\begin{gather}
\sum_{k=0}^\infty\frac{\binom{4k}{2k}}{2^{6k}}\big(4\cos^22\psi\big)^k\mathsf H_k\nonumber\\
\qquad{} =\frac{2\sqrt{2}}{\pi}\int_0^\infty\frac{\Li_1\left( \frac{4X^{4}\cos^22\psi}{(1+X^4)^{2}} \right)}{1-\frac{4X^{4}\cos^22\psi}{(1+X^4)^{2}}}\frac{X^{2}\D X}{1+X^4} =\frac{\log\frac{( \csc \psi +\sqrt{2})^2}{4 (\cot \psi +1)}}{ \sqrt{2}\sin\psi }+\frac{ \log\frac{( \sec \psi +\sqrt{2})^2}{4 (\tan \psi +1)}}{ \sqrt{2}\cos\psi}.\label{eq:4kHklog}
\end{gather}\end{enumerate}
\end{Theorem}

\begin{Remark}Picking $ \varphi=\pi/6$ in part (a) of the theorem above, we recover some entries of Remark~\ref{tab:3kSun'}. Plugging $ \varphi=2\pi/15$ into
\begin{align*}
\sum_{k=0}^\infty \frac{\binom{3k}k}{3^{3k}}\left( 4\cos^2\frac{3\varphi}{2} \right)^k(\mathsf H_{3k}-\mathsf H_{2k})=\frac{\cos \frac{\varphi }{2}+\cos \left(\frac{\varphi }{2}-\frac{\pi }{3}\right) }{\sqrt{3} \sin\frac{3 \varphi }{2}}\log\frac{3}{4 \cos\frac{\varphi }{2}\cos \left(\frac{\varphi }{2}-\frac{\pi }{3}\right)},
\end{align*} we arrive at \begin{align*}
\sum_{k=0}^\infty \binom{3k}k\left(\frac{3+\sqrt5}{54}\right)^k(\mathsf H_{3k}-\mathsf H_{2k})=\frac{1+\sqrt{5}}{2}\left(\log3-2\log\frac{1+\sqrt{5}}{2}\right),
\end{align*}as suggested by Sun \cite[formula~(2.15) in Conjecture~2.7]{Sun2022}.\end{Remark}

\begin{Remark}Together with the right-hand sides of \eqref{eq:4k2k_a}--\eqref{eq:4k2k_b} for $r=1$, part (b) of the theorem above answers a question of Sun \cite[formula~(2.16) in Conjecture~2.8]{Sun2022}.
\end{Remark}
\begin{Remark}For $ \psi\in\{\pi/8,\pi/12\}$, one can paraphrase \eqref{eq:4kHklog} as
\begin{align*}
&\sum_{k=0}^\infty\frac{\binom{4k}{2k}}{2^{5k}}\mathsf H_k= \sqrt{1+\frac{1}{\sqrt{2}}}\R\left[ 7\Li_1\big({\rm e}^{\pi {\rm i}/8}\big)+7\Li_1\big({\rm e}^{3\pi {\rm i}/8}\big)+7\Li_1\big({\rm e}^{5\pi {\rm i}/8}\big)+3\Li_1\big({\rm e}^{7\pi {\rm i}/8}\big)\right]\\
&\hphantom{\sum_{k=0}^\infty\frac{\binom{4k}{2k}}{2^{5k}}\mathsf H_k=}{}
+\sqrt{1-\frac{1}{\sqrt{2}}}\R\left[ \Li_1\big({\rm e}^{\pi {\rm i}/8}\big)-3\Li_1\big({\rm e}^{3\pi {\rm i}/8}\big)+\Li_1\big({\rm e}^{5\pi {\rm i}/8}\big)+\Li_1\big({\rm e}^{7\pi {\rm i}/8}\big)\right],
\\
&\sum_{k=0}^\infty\frac{3^k\binom{4k}{2k}}{2^{6k}}\mathsf H_k=2\R\bigl[ \Li_1\big({\rm e}^{\pi {\rm i}/6}\big)+\Li_1\big({\rm e}^{2\pi {\rm i}/3}\big)\bigr]+2\sqrt{3}\R\bigl[ 2\Li_1\big({\rm e}^{\pi {\rm i}/6}\big)+3\Li_1\big({\rm e}^{\pi {\rm i}/2}\big)\bigr],
\end{align*}
thus testifying to \eqref{eq:4kZr16} and \eqref{eq:4kZr24}.
\end{Remark}

\section[CMZV characterizations of certain series involving (2k choose k)]{CMZV characterizations of certain series involving $\boldsymbol{{ 2k \choose k}}$}\label{sec:centr_binom}

\subsection{Recursions and fibrations of polylogarithms}
Adopting
 the notational conventions of Frel\-les\-vig--Tom\-ma\-si\-ni--Wever \cite[Section~2]{Frellesvig2016}, we define Goncharov's generalized polylogarithm (GPL) recursively by an integral along a straight line segment\begin{align}
G(\alpha_{1},\dots,\alpha_n;z):= \int_0^z\frac{G(\alpha_2,\dots,\alpha_n;x)\D x}{x-\alpha_1}\label{eq:GPL_rec}
\end{align} for $ |\alpha_1|^2+\dots+|\alpha_n|^2\neq0$, with the extra settings that \begin{align}
G(\underset{m }{\underbrace{0,\dots,0 }};z):= \frac{\log^mz}{m!},\qquad G(-\!\!-;z):= 1.\label{eq:GPL_log_bd}
\end{align} These GPLs are related to MPLs [defined in \eqref{eq:Mpl_defn}] by the following equation (see \cite[formula~(1.3)]{Panzer2015} or \cite[formula~(2.6)]{Frellesvig2016})
\begin{align}
G\Big(\underset{a_1-1 }{\underbrace{0,\dots,0 }},\widetilde \alpha_1,\underset{a_2-1 }{\underbrace{0,\dots,0 }},\widetilde \alpha_2,\dots,\underset{a_n-1 }{\underbrace{0,\dots,0 }},\widetilde \alpha_n;z\Big)=(-1)^n\Li_{a_1,\dots,a_n}\left( \frac{z}{\widetilde \alpha_1} ,\frac{\widetilde \alpha_1}{\widetilde \alpha_2},\dots, \frac{\widetilde \alpha_{n-1}}{\widetilde \alpha_n}\right)\!\!\!\!\!\label{eq:GPL_MPL}
\end{align}
when $ \prod_{j=1}^n\widetilde \alpha_j\neq0$.
As
a result of this GPL-MPL correspondence, we may reformulate \eqref{eq:Zk(N)_defn} as \begin{align}
\mathfrak Z_{k}(N):= \Span_{\mathbb Q}\bigl\{G(z_1,\dots,z_k;1)\mid z_1^N,\dots,z_{k}^N\in\{0,1\}, \,   z_1\neq1,\, z_{k}\neq0 \bigr\}.\tag{\ref{eq:Zk(N)_defn}$'$}\label{eq:Zk(N)_defn'}
\end{align}
Here, the convergence of $G(z_1,\dots,z_k;1)$ requires that $ z_1\neq1$, while \eqref{eq:GPL_MPL} stipulates that $ z_k\neq0$.

Echoing the GPL recursion \eqref{eq:GPL_rec} for a special class of MPLs
\begin{align*}
\Li_k(z)=-G(\underset{k-1 }{\underbrace{0,\dots,0 }},1;z),
\end{align*}
we analytically continue polylogarithms $ \Li_k(z):= \sum_{n=1}^\infty\frac{z^k}{n^k},|z|<1$
[cf.\ \eqref{eq:Mpl_defn}] to all $ z\in\mathbb C\smallsetminus[1,\infty)$, by
\begin{align*}
& \R \Li_1(z) :=  -\log|1-z|, \\
& \I\Li_1(z) := -\arg(1-z)\in[-\pi,\pi),\ \\
& \Li_{k+1}(z) := \int_0^z \Li_k(w)\frac{\D w}{w},\qquad k\in\mathbb Z_{>0},
\end{align*}
where the integration path is also a straight line segment. With the recursive definition of~polylogarithms,
one can show inductively that
\begin{align}
\Li_k(x+{\rm i} 0^{+})-\Li_k(x-{\rm i}0^{+})=\frac{2\pi {\rm i}}{(k-1)!}\log^{k-1} x\label{eq:Lik_jump}
\end{align}
holds for all $x\in(1,\infty)$ and $k\in\mathbb Z_{>0} $.

In the rest of this note, we will frequently need to integrate the differential $ 1$-form\begin{align*}
G(\alpha_{1}(t),\dots,\alpha_n(t);z(t))\frac{\D t}{R(t)}
\end{align*}over certain contours, where $ \alpha_{1}(t),\dots,\alpha_n(t)$, $z(t)$, $R(t)$ are rational functions of $t$.
To fully exploit the recursion \eqref{eq:GPL_rec} while evaluating such contour integrals, we take fibration procedures \cite[Lemma 2.14 and Corollary 3.2]{Panzer2015} to rewrite $ G(\alpha_{1}(t),\dots,\alpha_n(t);z(t))$ in terms of finitely many new GPLs in the shape of $ G(\beta_1,\dots,\beta_n;t)$, where the constants $ \beta_1,\dots,\beta_n$ are sorted from the zeros and poles of the rational functions $ z(t)-\alpha_j(t)$, $\alpha_j(t)-\alpha_{j'}(t)$, $\alpha_n(t)$ for $ j,j'\in\mathbb Z\cap[1,n]$.

For example, the branch cut analysis in \cite[Corollary 3.8\,(a)]{Zhou2022mkMpl} brings us the following fibration structure for $ r\in\mathbb Z_{>1}$ and $ 4|\chi|\leq |1+\chi|^{2}$:
\begin{gather}
\ointctrclockwise_{|z|=1}\frac{\Li_r\left( \frac{\chi(1+z)^{2}}{z(1+\chi)^{2}} \right)}{1- \frac{\chi(1+z)^{2}}{z(1+\chi)^{2}}}\frac{\D z}{2\pi {\rm i}z}=\frac{1+\chi}{1-\chi}\ointctrclockwise_{|z|=1}\left( \frac{1}{z-\chi} -\frac{1}{z-\frac{1}{\chi}}\right)\Li_r\left( \frac{\chi(1+z)^{2}}{z(1+\chi)^{2}} \right)\frac{\D z}{2\pi {\rm i}}\nonumber\\
\qquad\in\frac{1+\chi}{1-\chi}\Span_{\mathbb Q}\big\{(\pi {\rm i})^{r-\ell}G(\alpha_1,\dots,\alpha_\ell;\chi)\in\mathbb C \mid \alpha_1^{2},\dots,\alpha _\ell^{2}\in\{0,1\},\,\ell\in\mathbb Z\cap[0,r]\big\}\nonumber\\
\qquad=: \frac{1+\chi}{1-\chi}\mathfrak H_{r}^{(\chi)}(2),\label{eq:chi_fib}
\end{gather}
while the argument in \cite[Corollary 3.8\,(b)]{Zhou2022mkMpl} leads us to the following relation for $ r\in\mathbb Z_{>1}$ and~$2|\upsilon|\leq |1+\upsilon^{2}|$:
\begin{gather}
\ointctrclockwise_{|z|=1}\frac{\Li_r\left( \pm\frac{\upsilon(1+z^{2})}{z(1+\upsilon^{2})} \right)}{1\mp\frac{\upsilon(1+z^{2})}{z(1+\upsilon^{2})}}\frac{\D z}{2\pi {\rm i}z}=\frac{1+\upsilon^2}{1-\upsilon^2}\ointctrclockwise_{|z|=1}\left( \frac{1}{z\mp\upsilon} -\frac{1}{z\mp\frac{1}{\upsilon}}\right)\Li_r\left( \pm\frac{\upsilon(1+z^{2})}{z(1+\upsilon^{2})} \right)\frac{\D z}{2\pi {\rm i}}\nonumber\\
\qquad\in\frac{1+\upsilon^2}{1-\upsilon^2}\Span_{\mathbb Q}\big\{(\pi {\rm i})^{r-\ell}G(\alpha_1,\dots,\alpha_\ell;\upsilon)\in\mathbb C \mid \alpha_1^{4},\dots,\alpha _\ell^{4}\in\{0,1\},\,\ell\in\mathbb Z\cap[0,r]\big\}\nonumber\\
\qquad=: \frac{1+\upsilon^2}{1-\upsilon^2}\mathfrak H_{r}^{(\upsilon)}(4).\label{eq:ups_fib}
\end{gather}
Here in both \eqref{eq:chi_fib} and \eqref{eq:ups_fib}, the integrands are pole-free when the contours are shrunk to tight loops wrapping around the branch cuts of polylogarithms,
because the jump discontinuities in~\eqref{eq:Lik_jump} tend to zero as $ x\to1+0^+$, for $ k\in\mathbb Z_{>1}$.

\subsection{Some CMZVs of levels 8 and 12\label{subsec:CMZV_8_12}}With the foregoing preparations, we are ready for the verification of our claims in \eqref{eq:2kZ8}--\eqref{eq:2kZ12'}.

\begin{proof}[Proof of Theorem \ref{thm:2k}]Recall that \cite[Section~5.1]{Mezo2014}\begin{align}
\sum_{k=1}^\infty w^k \mathsf H_{k}^{(r)}=\frac{\Li_{r}(w)}{1-w}\label{eq:H_gf}
\end{align}holds for $ |w|<1$, $r\in\mathbb Z_{>0}$, and \begin{align*}
\ointctrclockwise_{|z|=1}\left[ \frac{(1+z)^2}{z} \right]^k\frac{\D z}{2\pi {\rm i}z}=\frac{1}{2\pi}\int_0^{2\pi}\left( 4\cos^2\frac{\theta}{2} \right)^k\D\theta={\binom{2k}k}.
\end{align*}Thus, for $ 4|\chi|\leq |1+\chi|^{2}$, the relation\begin{align}
\sum_{k=0}^\infty\binom{2k} k\left[ \frac{\chi}{(1+\chi)^{2}} \right]^{k}\mathsf H_k^{(r)}={}&\displaystyle \ointctrclockwise_{|z|=1}\frac{\Li_r\left( \frac{\chi(1+z)^{2}}{z(1+\chi)^{2}} \right)}{1-\frac{\chi(1+z)^{2}}{z(1+\chi)^{2}}}\frac{\D z}{2\pi {\rm i}z}\in\frac{1+\chi}{1-\chi}\mathfrak H_{r}^{(\chi)}(2)\tag{\ref{eq:chi_fib}$'$}\label{eq:chi_fib''}
\end{align} is immediate from termwise integration and reference to \eqref{eq:chi_fib}. Likewise, from\begin{align}
\sum_{k=1}^\infty w^{2k}\mathsf H_{2k}^{(r)}=\frac{\Li_{r}(w)}{1-w}+\frac{\Li_{r}(-w)}{1+w}\label{eq:H_gf'}
\end{align} for $ |w|<1$, $r\in\mathbb Z_{>0}$, one deduces the following relation for $ 2|\upsilon|\leq |1+\upsilon^{2}| $:
\begin{align}
\sum_{k=0}^\infty\binom{2k} k\left( \frac{\upsilon}{1+\upsilon^{2}} \right)^{2n}\mathsf H_{2k}^{(r)}={}&\ointctrclockwise_{|z|=1}\left[\frac{\Li_r\left( \frac{\upsilon(1+z^{2})}{z(1+\upsilon^{2})} \right)}{1-\frac{\upsilon(1+z^{2})}{z(1+\upsilon^{2})}}+\frac{\Li_r\left(- \frac{\upsilon(1+z^{2})}{z(1+\upsilon^{2})} \right)}{1+\frac{\upsilon(1+z^{2})}{z(1+\upsilon^{2})}}\right]\frac{\D z}{2\pi {\rm i}z}\nonumber\\
&\in{}\frac{1+\upsilon^2}{1-\upsilon^2}\mathfrak H_{r}^{(\upsilon)}(4),
\tag{\ref{eq:ups_fib}$'$}\label{eq:ups_fib''}
\end{align}
upon referring to \eqref{eq:ups_fib}. Thus, with specially chosen parameters \begin{align*}
\frac{\chi}{(1+\chi)^{2}}=\begin{cases}\dfrac{1}{8}, & \chi=\big(\sqrt{2}-1\big)^2,\vspace{1mm}\\
\dfrac{1}{16}, & \chi=\big(2-\sqrt{3}\big)^2, \\
\end{cases}\qquad\left( \frac{\upsilon}{1+\upsilon^{2}} \right)^{2}=\begin{cases}\dfrac{1}{8}, & \upsilon=\sqrt{2}-1,\vspace{1mm} \\
\dfrac{1}{16}, & \upsilon=2-\sqrt{3}, \\
\end{cases}
\end{align*} we have confirmed the equalities in \eqref{eq:2kZ8}--\eqref{eq:2kZ12'}.

To build the remaining halves of \eqref{eq:2kZ8}--\eqref{eq:2kZ12'} on the knowledge of \eqref{eq:chi_fib''} and \eqref{eq:ups_fib''}, we notice that
\begin{align*}
\frac{1+\chi}{1-\chi}=\begin{cases}\sqrt{2}, & \chi=\big(\sqrt{2}-1\big)^2, \\
\dfrac{2}{\sqrt{3}}, & \chi=\big(2-\sqrt{3}\big)^2,
\end{cases}
\qquad \frac{1+\upsilon^2}{1-\upsilon^2}=\begin{cases}\sqrt{2}, & \upsilon=\sqrt{2}-1, \\
\dfrac{2}{\sqrt{3}}, & \upsilon=2-\sqrt{3},
\end{cases}
\end{align*} while the \texttt{IterIntDoableQ} function in Au's \texttt{MultipleZetaValues} package (v1.2.0) \cite{Au2022a} certifies that
\begin{align*}
& \mathfrak H_{r}^{(\chi)}(2)\subset  \begin{cases}\mathfrak Z_r(8), & \chi=\big(\sqrt{2}-1\big)^2, \\
\mathfrak Z_r(12), & \chi=\big(2-\sqrt{3}\big)^2,
\end{cases}
\qquad \mathfrak H_{r}^{(\upsilon)}(4)\subset \begin{cases}\mathfrak Z_r(8), & \upsilon=\sqrt{2}-1, \\
\mathfrak Z_r(12), & \upsilon=2-\sqrt{3}.
\end{cases}
\end{align*}
 Au's abovementioned package produces analytic evaluations of integrals like \eqref{eq:2kZ8}--\eqref{eq:2kZ12'}, so long as the results are in the $ \mathbb Q$-vector spaces $ \mathfrak Z_k(8)$, $k\in\{1,2,3,4\}$ or $ \mathfrak Z_k(12)$, $k\in\{1,2,3\}$, hence Remark~\ref{tab:lin_binom_Sun}. \end{proof}

In the proof above and some arguments in the next section, we are effectively dealing with special cases of Au's integrals \cite[Section~3]{Au2022a}
\begin{align}\int_a^b {R_0(x)} \Li_{r_1}(R_1(x))\cdots \Li_{r_k}(R_k(x))\D x,\label{eq:Au_intLi}\end{align}
where the 1-form $R_0(x)\D x$ has only simple poles on the complex sphere, and the functions~$ R_j(x)$, $1\leq j\leq k$, are rational. One may find out the structure of \eqref{eq:Au_intLi} by looking at a finite set~$ S$ that is the union of~$ \{a,b,\infty\}$ with $ R_j^{-1}(0)$, $ R_j^{-1}(1)$, $ R_j^{-1}(\infty)$ for $1\leq j\leq k$ and all the poles of~$R_0(x)\D x$. If $S$ is a subset of $ \{0,\infty\}\cup\big\{{\rm e}^{2\pi {\rm i} m/N} \big| m\in\mathbb Z\cap[1,N]\big\}$, then \eqref{eq:Au_intLi} is expressible through CMZVs of level $N$. In some other situations, feeding the set $S\smallsetminus\{\infty\} $ to the \texttt{IterIntDoableQ} function in Au's \texttt{MultipleZetaValues} package (v1.2.0) \cite{Au2022a}, one may also receive a CMZV level $N\leq12$ that is compatible with our analysis. However, residue calculus is still required in order to obtain the correct $ \overline{\mathbb Q}$-linear combinations of CMZVs in Theorems \ref{thm:2k}, \ref{thm:3k4k}, and \ref{thm:logSun}, and to eliminate extraneous factors of $ \frac{1}{2\pi {\rm i}}$.

\section[CMZV and logarithmic characterizations of certain series involving (3k choose k) and (4k choose 2k)]{CMZV and logarithmic characterizations of certain series\\ involving $\boldsymbol{{{3k}\choose k}}$ and $\boldsymbol{{{4k}\choose {2k}}}$}\label{sec:non-centr_binom}

\subsection[Integral representations of (3k choose  k)\^{}\{pm1\} and (4k choose 2k)]{Integral representations of $\boldsymbol{{{3k}\choose {k}}^{\pm1}}$ and $\boldsymbol{{{4k}\choose {2k}}}$}
It is an elementary exercise in Euler's (di)gamma functions that
\begin{gather*}
\frac{1}{k\binom{3k}k}= \int_0^1\big[t(1-t)^2\big]^k\frac{\D t}{t},\\
\frac{\mathsf H_{3k}-\mathsf H_k+\frac{1}{k}}{k\binom{3k}k}= -\int_0^1\big[t(1-t)^2\big]^k\log t\frac{\D t}{t},\\
\frac{\mathsf H_{3k}-\mathsf H_{2k}}{k\binom{3k}k}= -\int_0^1\big[t(1-t)^2\big]^k\log (1-t)\frac{\D t}{t}.
\end{gather*}
These identities support the integral representations of infinite series in \eqref{eq:3kZ4}--\eqref{eq:3kZ4''}, in conjunction with the generating function \eqref{eq:H_gf} for harmonic numbers.

One can verify\begin{align}\frac{2\pi}{3\sqrt{3}}
\frac{ \binom{3 k}{k}}{ 3^{3 k}}=\int_0^\infty\left[ \frac{X^3}{(1+X^{3})^2} \right]^{k}\frac{\D X}{1+X^3}\label{eq:3k_int_repn}
\end{align} and its equivalent form \big(up to a variable substitution $X=x^2$\big)
\begin{align}
\frac{\pi}{3\sqrt{3}}
\frac{ \binom{3 k}{k}}{ 3^{3 k}}=\int_0^\infty\left( \frac{x^3}{1+x^{6}} \right)^{2k}\frac{x\D x}{1+x^6}\tag{\ref{eq:3k_int_repn}$'$}\label{eq:3k_int_repn'}
\end{align}
by manipulating an integral representation for Euler's beta function \cite[item 8.380.3]{GradshteynRyzhik}. The~deri\-va\-tive of \eqref{eq:3k_int_repn} with respect to $k$ leaves us \begin{align}\frac{2\pi}{3\sqrt{3}}
\frac{ \binom{3 k}{k}}{ 3^{3 k}}(
3\mathsf H_{3 k}-2\mathsf H_{2 k}-\mathsf H_k-3 \log 3)=\int_0^\infty\left[ \frac{X^3}{(1+X^{3})^2} \right]^{k}\frac{\log\frac{X^3}{(1+X^{3})^2}\D X}{1+X^3}.\label{eq:3k_int_repn_H}
\end{align}Therefore, the generating function \eqref{eq:H_gf} brings us the equalities between series and integrals in \eqref{eq:3kHkZr}--\eqref{eq:3kH3kZ1} and \eqref{eq:3kHklog}--\eqref{eq:3kH2klog}.
In fact, as long as the inequalities $ \big|1+\varXi^3\big|^2<4\big|\varXi^3\big|$ and $ \big|1+\xi^6\big|<2\big|\xi^3\big|$ are honored, we have the following extensions (cf.\ \eqref{eq:H_gf}, \eqref{eq:H_gf'}, \eqref{eq:3k_int_repn}, \eqref{eq:3k_int_repn'}, and~\eqref{eq:3k_int_repn_H}) of \eqref{eq:3kHklog}--\eqref{eq:3kH2klog}:
\begin{gather*}
\sum_{k=0}^\infty\frac{\binom{3k}k}{3^{3k}}\left[ \frac{\big(1+\varXi^3\big)^2}{\varXi^3} \right]^k\mathsf H_k^{(r)}=\frac{3\sqrt{3}}{2\pi}\int_0^\infty\frac{\Li_r\left( \frac{X^{3}(1+\varXi^3)^{2}}{\varXi^3(1+X^3)^{2}} \right)}{1-\frac{X^{3}(1+\varXi^3)^{2}}{\varXi^3(1+X^3)^{2}}}\frac{\D X}{1+X^3},
\\
\sum_{k=0}^\infty\frac{\binom{3k}k}{3^{3k}}\left[ \frac{\big(1+\varXi^3\big)^2}{\varXi^3} \right]^k\big(
3\mathsf H_{3 k}^{} -2\mathsf H_{2 k}^{}-\mathsf H_k^{}-3 \log 3\big)\mathsf H_k^{(r)}\nonumber\\
\qquad{} =
\frac{3\sqrt{3}}{2\pi}\int_0^\infty\frac{\Li_r\left( \frac{X^{3}(1+\varXi^3)^{2}}{\varXi^3(1+X^3)^{2}} \right)}{1-\frac{X^{3}(1+\varXi^3)^{2}}{\varXi^3(1+X^3)^{2}}}\frac{\log\frac{X^{3}}{(1+X^3)^{2}}\D X}{1+X^3},
\\
\sum_{k=0}^\infty\frac{\binom{3k}k}{3^{3k}}\left( \frac{1+\xi^{6}}{\xi^3} \right)^{2k}\mathsf H_{2k}^{(r)}=\frac{3\sqrt{3}}{2\pi}\int_0^\infty\left[\frac{\Li_r\left( \frac{x^{3}(1+\xi^{6})}{\xi^3(1+x^6)} \right)}{1-\frac{x^{3}(1+\xi^{6})}{\xi^3(1+x^6)}}+\frac{\Li_r\left(- \frac{x^{3}(1+\xi^{6})}{\xi^3(1+x^6)} \right)}{1+\frac{x^{3}(1+\xi^{6})}{\xi^3(1+x^6)}}\right]\frac{x\D x}{1+x^6} 
\end{gather*}
for $ r\in\mathbb Z_{>0}$.

Akin to \eqref{eq:3k_int_repn}, we have \begin{align*}
\frac{\pi}{2\sqrt{2}}
\frac{ \binom{4 k}{2k}}{ 2^{6 k}}=\int_0^\infty\left[ \frac{X^4}{(1+X^{4})^2} \right]^{k}\frac{X^{2}\D X}{1+X^4},
\end{align*}so termwise integration of the generating function \eqref{eq:H_gf} steers us to \begin{align*}
\sum_{k=0}^\infty\frac{\binom{4k}{2k}}{2^{6k}}\left[ \frac{\big(1+\varXi^4\big)^2}{\varXi^4} \right]^k\mathsf H_k^{(r)}=\frac{2\sqrt{2}}{\pi}\int_0^\infty\frac{\Li_r\left( \frac{X^{4}(1+\varXi^4)^2}{\varXi^4(1+X^4)^{2}} \right)}{1-\frac{X^{4}(1+\varXi^4)^2}{\varXi^4(1+X^4)^{2}}}\frac{X^{2}\D X}{1+X^4}
\end{align*} for $ \big|1+\varXi^4\big|^2<4\big|\varXi^4\big|$ and $ r\in\mathbb Z_{>0}$, which proves \textit{a fortiori} the equalities in \eqref{eq:4kZr16}--\eqref{eq:4kZr24} and~\eqref{eq:4kHklog}.

\subsection{Some CMZVs of levels 4, 12, 16, and 24\label{subsec:CMZVs}}We start
by treating the CMZVs of level $4$ appearing in \eqref{eq:3kZ4}--\eqref{eq:3kZ4''}.
\begin{proof}[Proof of Theorem \ref{thm:3k4k}\,(a)]
First we demonstrate the following fibration structure for $ r\in\mathbb Z_{>0}$:\begin{align}
\Li_r\left( \frac{t(1-t)^2}{2} \right)\in\Span_{\mathbb Z} \{ G(\beta_1,\dots,\beta_r;t) \mid \beta_j\in\{0,1\},\, j\in\mathbb Z\cap[1,r); \, \beta_r\in\{2,{\rm i},-{\rm i}\}  \}.\!\label{eq:2i-i01}
\end{align}Clearly, the claim is true for \begin{align*}
\Li_1\left( \frac{t(1-t)^2}{2} \right)=-G(2;t)-G({\rm i};t)-G(-{\rm i};t).
\end{align*}From\begin{align*}
\Li_r\left( \frac{t(1-t)^2}{2} \right)=\int_0^t\left( \frac{1}{s}+\frac{2}{s-1} \right)\Li_{r-1}\left( \frac{s(1-s)^2}{2} \right)\D s
\end{align*}and the GPL recursion \eqref{eq:GPL_rec}, we may verify \eqref{eq:2i-i01} inductively.

Appealing again to the GPL recursion \eqref{eq:GPL_rec}, we see that
\begin{gather*}
\int_0^1\Li_r\left( \frac{t(1-t)^2}{2} \right)\frac{\D t}{t}\\
\qquad{} \in \Span_{\mathbb Z}\{  G(0,\beta_1,\dots,\beta_r;1)\mid
  \beta_j\in\{0,1\},j\in\mathbb Z\cap[1,r);\,\beta_r\in\{2,{\rm i},-{\rm i}\}\}.
\end{gather*}
Here, it is evident from the definition of $ \mathfrak Z_k(N)$ in \eqref{eq:Zk(N)_defn'} that \begin{align*}
\Span_{\mathbb Z}\{ G(0,\beta_1,\dots,\beta_r;1)\mid \beta_j\in\{0,1\},j\in\mathbb Z\cap[1,r);\, \beta_r\in\{{\rm i},-{\rm i}\}\}\subset\mathfrak Z_{r+1}(4).
\end{align*}Meanwhile, either a variable substitution in \eqref{eq:GPL_rec} or a special case of the H\"older relation \cite[formula~(2.13)]{Frellesvig2016} leads us to $G(0,\beta_1,\dots,\beta_{r-1},2;1)=(-1)^{r+1}G(-1,1-\beta_{r-1},\dots,1-\beta_1,1;1)$, so it follows that \begin{align*}
\Span_{\mathbb Z}\{ G(0,\beta_1,\dots,\beta_{r-1},2;1)\mid \beta_j\in\{0,1\},j\in\mathbb Z\cap[1,r) \}\subset\mathfrak Z_{r+1}(2)\subseteq\mathfrak Z_{r+1}(4).
\end{align*} This completes the verification of \eqref{eq:3kZ4}.

Before extending such a service to \eqref{eq:3kZ4'} and \eqref{eq:3kZ4''}, we note that the shuffle product of GPLs \cite[formula~(2.4)]{Frellesvig2016} implies\begin{align}
G(\alpha;t)G(\beta_1,\dots,\beta_r;t)={}&G(\alpha,\beta_1,\dots,\beta_r;t)+\sum_{j=1}^{r-1}G(\beta_{1},\dots,\beta_j,\alpha,\beta_{j+1},\dots ,\beta_r;t)\nonumber\\{}&+G(\beta_1,\dots,\beta_r,\alpha;t),\label{eq:ShGPL}
\intertext{and especially}
0=G(0;1)G(\beta_1,\dots,\beta_r;1)={}&G(0,\beta_1,\dots,\beta_r;1)+\sum_{j=1}^{r-1}G(\beta_{1},\dots,\beta_j,0,\beta_{j+1},\dots ,\beta_r;1)\nonumber\\{}&{}+G(\beta_1,\dots,\beta_r,0;1)\tag{\ref{eq:ShGPL}$'$}\label{eq:ShGPL'}
\end{align}for $ G(0;1)=\log 1=0$ [cf.\ \eqref{eq:GPL_log_bd}]. In other words, an extra factor of $ \log t=G(0;t)$ or $ \log(1-t)=G(1;t)$ in the integrand of \eqref{eq:3kZ4'} or \eqref{eq:3kZ4''} modifies the fibration structure on the right-hand side of \eqref{eq:2i-i01} into $\beta_1,\dots,\beta_{r+1}\in\{0,1, {\rm i},-{\rm i}\} $ or $\beta_1,\dots,\beta_{r+1}\in\{0,1,2\} $, and $ \beta_{r+1}\neq0$.
This accounts for the CMZV characterizations in \eqref{eq:3kZ4'} and \eqref{eq:3kZ4''}. \end{proof}

As we have just seen, the fibration structure of GPLs plays pivotal r\^oles in the CMZV characterization at level $N=4$. The same will apply to the $ N\in\{12,24\}$ situations at the next stage.

For technical reasons related to the branch cuts of polylogarithms [see \eqref{eq:Lik_jump}], we will focus on the $ \mathfrak Z_k(N),k\in\mathbb Z_{>1}$ cases in this subsection,
deferring the treatment of $ \mathfrak Z_1(N)$ to Section~\ref{subsec:logSun}.

\begin{proof}[Proof of Theorem \ref{thm:3k4k}\,(b) for $\boldsymbol{\mathfrak Z_k(12)}$ and $\boldsymbol{\mathfrak Z_{k}(24)}$,
where $\boldsymbol{k\in\mathbb Z_{>1}}$] Begin with an  equivalent formulation of the equality in \eqref{eq:3kHkZr}:
\begin{align}
\sum_{k=0}^\infty\frac{2^k\binom{3k}k}{3^{3k}}\mathsf H_k^{(r)}={}&\frac{3\sqrt{3}}{2\pi}\frac{1}{ 1-{\rm e}^{2\pi {\rm i}/3}}\left(\int_{\infty {\rm e}^{2\pi {\rm i}/3}}^0+\int_0^\infty\right)\frac{\Li_r\big( \frac{2X^{3}}{(1+X^3)^{2}} \big)}{1-\frac{2X^{3}}{(1+X^3)^{2}}}\frac{\D X}{1+X^3},\tag{\ref{eq:3kHkZr}$'$}\label{eq:3kHkZr'}
\end{align}
where the contour consists of two rays passing through the origin in the complex $X $-plane, whose angles of inclination are $ 120^\circ$ and $ 0^\circ$ respectively. Had we traded $ X$ for $1/X$ in \eqref{eq:3kHkZr} before such a conversion, we would have ended up with
\begin{align}
\sum_{k=0}^\infty\frac{2^k\binom{3k}k}{3^{3k}}\mathsf H_k^{(r)}={}&\frac{3\sqrt{3}}{2\pi}\frac{1}{ 1-{\rm e}^{4\pi {\rm i}/3}}\left(\int_{\infty {\rm e}^{2\pi {\rm i}/3}}^0+\int_0^\infty\right)\frac{\Li_r\big( \frac{2X^{3}}{(1+X^3)^{2}} \big)}{1-\frac{2X^{3}}{(1+X^3)^{2}}}\frac{X\D X}{1+X^3}\tag{\ref{eq:3kHkZr}$''$}\label{eq:3kHkZr''}
\end{align}
instead. Averaging over \eqref{eq:3kHkZr'} and \eqref{eq:3kHkZr''}, we arrive at \begin{align}
\sum_{k=0}^\infty\frac{2^k\binom{3k}k}{3^{3k}}\mathsf H_k^{(r)}=\left(\int_{\infty {\rm e}^{2\pi {\rm i}/3}}^0+\int_0^\infty\right)\Li_r\left( \frac{2z^{3}}{(1+z^3)^{2}} \right) R_{12}(z)\frac{\D z}{2\pi {\rm i}},\tag{\ref{eq:3kHkZr}$^*$}\label{eq:2kHkZr_L}
\end{align}where the rational function\begin{align*}
R_{12}(z):= \frac{3\sqrt{3}{\rm i}}{2}\left(\frac{1}{ 1-{\rm e}^{2\pi {\rm i}/3}}+\frac{z}{1-{\rm e}^{4\pi {\rm i}/3}}\right)\frac{1}{1-\frac{2z^{3}}{(1+z^3)^{2}}}\frac{1}{1+z^{3}}
\end{align*}has only six simple poles located at $ {\rm e}^{(2m-1)\pi {\rm i}/6}$, $m\in\mathbb Z\cap[1,6]$, with all the corresponding residues being in the $ \mathbb Q$-vector space $ \mathbb Q+\mathbb Q\sqrt 3$. Thanks to Cauchy's integral formula, we obtain\begin{align}
\sum_{k=0}^\infty\frac{2^k\binom{3k}k}{3^{3k}}\mathsf H_k^{(r)}=\lim_{\varepsilon\to0^+}\int_{C_\varepsilon^{[\pi/6,\pi/2]}}\Li_r\left( \frac{2z^{3}}{(1+z^3)^{2}} \right) R_{12}(z)\frac{\D z}{2\pi {\rm i}},\tag{\ref{eq:3kHkZr}$^{**}$}\label{eq:3kHkZr_Ceps}
\end{align}where the counterclockwise contour $C_\varepsilon^{[\alpha,\beta]} $ is separated from the circular arc $ {\rm e}^{{\rm i}\theta}$, $\theta\in[\alpha,\beta]$ by a distance $\varepsilon$. In other words, we will need to focus on the branch cut discontinuities of the integrand in \eqref{eq:3kHkZr_Ceps}, which satisfy [cf.\ \eqref{eq:Lik_jump} and \eqref{eq:ShGPL}]
\begin{gather*}
\lim_{\varepsilon\to0^+}
\left[\Li_r\left( \frac{2\big[(1+\varepsilon){\rm e}^{{\rm i}\theta}\big]^{3}}{\big\{1+\big[(1+\varepsilon){\rm e}^{{\rm i}\theta}\big]^{3}\big\}^{2}} \right)-\Li_r\left( \frac{2\big[(1-\varepsilon){\rm e}^{{\rm i}\theta}\big]^{3}}{\big\{1+\big[(1-\varepsilon){\rm e}^{{\rm i}\theta}\big]^{3}\big\}^{2}} \right)\right]\\
\qquad{}
=
\pi {\rm i}\frac{\theta-\frac{\pi}{3}}{\bigl\vert \theta-\frac{\pi}{3}\bigr\vert}\frac{\log^{r-1}\frac{2{\rm e}^{3{\rm i}\theta}}{(1+{\rm e}^{3{\rm i}\theta})^{2}}}{(r-1)!}
\\
\qquad{}
\in 2\pi {\rm i}\frac{\theta-\frac{\pi}{3}}{\bigl\vert \theta-\frac{\pi}{3}\bigr\vert}\Span_{\mathbb Q}\big\{G\big(\beta_1,\dots,\beta_{r-1-\ell};{\rm e}^{{\rm i}\theta}\big)\log^\ell2\mid  \\
\qquad\quad{} \beta_j^{3}\in\{0,-1\},\, j\in\mathbb Z\cap[1,r-1-\ell];\, \ell\in\mathbb Z\cap[0,r-1]  \big\}
\end{gather*}
for
$ \theta\in[\pi/6,\pi/3)\cup(\pi/3,\pi/2]$.
Bearing in mind the filtration property $ \mathfrak Z_j(N)\mathfrak
Z_{ k}(N)\subseteq \mathfrak
Z_{j+ k}(N)$ for $ j,k\in \mathbb Z_{\geq0}$ \cite[Section~1.2]{Goncharov1998}, along with the fact that $ \log 2=-\Li_1(-1)\in\mathfrak Z_1(2)\subset\mathfrak Z_1(12)$, we may turn \eqref{eq:3kHkZr_Ceps} into
\begin{gather*}
\sum_{k=0}^\infty\frac{2^k\binom{3k}k}{3^{3k}}\mathsf H_k^{(r)}\in  \Span_{\mathbb Q}\Bigg\{\lim_{\varepsilon\to0^+}\left[ \int_{(1-\varepsilon){\rm e}^{\pi {\rm i}/6}} ^{(1-\varepsilon){\rm e}^{\pi {\rm i}/2}}- \int_{(1+\varepsilon){\rm e}^{\pi {\rm i}/6}} ^{(1+\varepsilon){\rm e}^{\pi {\rm i}/2}}\right]R_{12}(z) \\
\quad{} \times \frac{\arg z-\frac{\pi}{3}}{\bigl|\arg z-\frac{\pi}{3}\bigr|}G(\beta_1,\dots,\beta_{r-1-\ell};z)\log^\ell2\D z\Bigm| \\
\qquad  {}
\beta_{j}^{12}\in\{0,1\},\, j\in\mathbb Z\cap[1,r-1-\ell];\, \ell\in\mathbb Z\cap[0,r-1] \Bigg\}\\
\quad {}\subseteq \Span_{\mathbb Q+\mathbb Q\sqrt{3}}\big\{ G(\beta_1,\dots,\beta_{r-\ell};z)\log^\ell2\mid \\
\qquad{} z^{12}=1;\, \beta_{j}^{12}\in\{0,1\},\, j\in\mathbb Z\cap[1,r-\ell];\, \ell\in\mathbb Z\cap[0,r-1]\big\} \subseteq \mathfrak Z_{r}(12)+\sqrt{3}\mathfrak Z_{r}(12)
\end{gather*}
upon integrating over the branch cut discontinuities across a circular arc, via Panzer's logarithmic regularizations \cite[Section~2.3]{Panzer2015} that cancel out $ O(\log^{r}\varepsilon)$ contributions from individual GPL recursions \eqref{eq:GPL_rec}. Here, before arriving at the set inclusion
in the last step, we have used the scaling property \cite[formula~(2.3)]{Frellesvig2016} of GPLs $ G(\mu\alpha_1,\dots,\mu\alpha_n;\mu z)=G(\alpha_1,\dots,\alpha_n;z)$ for $\mu\neq0$, $\alpha_n\neq0 $, the shuffle algebra of GPLs
(see \eqref{eq:ShGPL} and \eqref{eq:ShGPL'} above), as well as the relation $ \pi {\rm i}\in\mathfrak Z_1(N)$ for $ N\in\mathbb Z_{\geq3}$ \cite[Lemma 4.1]{Au2022a}, to confirm that \begin{align}\mathfrak
Z_{k}(N)=\Span_{\mathbb Q}\big\{G(z_1,\dots,z_k;z) \mid z_1^N,\dots,z_{k}^N\in\{0,1\},\, z^{N}=1,\, z_1\neq z\big\}\tag{\ref{eq:Zk(N)_defn}$''$}\label{eq:Zk(N)_defn''}
\end{align}for $ N\in\mathbb Z_{\geq3}$.
So far, we have verified \eqref{eq:3kHkZr} in its entirety.

To prove \eqref{eq:3kH3kZr}, we need to deform the contour in\begin{gather}
\sum_{k=0}^\infty\frac{2^k\binom{3k}k}{3^{3k}}\big(
3\mathsf H_{3 k}^{}-2\mathsf H_{2 k}^{}-\mathsf H_k^{}+\log 2-3 \log 3\big)\mathsf H_k^{(r)}\nonumber\\
\qquad=\left(\int_{\infty {\rm e}^{2\pi {\rm i}/3}}^0+\int_0^\infty\right)\left[\Li_r\left( \frac{2z^{3}}{(1+z^3)^{2}} \right) \log \frac{2z^{3}}{(1+z^3)^{2}}\right]R_{12}(z)\frac{\D z}{2\pi {\rm i}}.\tag{\ref{eq:3kH3kZr}$^*$}\label{eq:3kH3kZr_star}
\end{gather}Like \eqref{eq:chi_fib} and \eqref{eq:ups_fib}, the integrand in \eqref{eq:3kH3kZr_star} is
pole-free when the new contour wraps around the branch cuts tightly.
In addition to an integral over the loop $ C_\varepsilon^{[\pi/6,\pi/2]}$ as in the last paragraph, we also need to take into account the jump discontinuities attributed to the factor $ \log \frac{2z^{3}}{(1+z^3)^{2}}$, which are located at a ray running from $0$ to $ \infty {\rm e}^{\pi {\rm i}/3}$. Exploiting the symmetry of the integral under a reflection $ z\mapsto {\rm e}^{2\pi {\rm i}/3}/z$, we can check that \begin{gather*}
\left( \int_{\delta {\rm e}^{\pi {\rm i}/3-{\rm i}0^{+}}}^{\frac{1}{\delta}{\rm e}^{\pi {\rm i}/3-{\rm i}0^{+}}}-\int_{\delta {\rm e}^{\pi {\rm i}/3+{\rm i}0^{+}}}^{\frac{1}{\delta}{\rm e}^{\pi {\rm i}/3+{\rm i}0^{+}}}
\right)\left[\Li_r\left( \frac{2z^{3}}{(1+z^3)^{2}} \right) \log \frac{2z^{3}}{(1+z^3)^{2}}\right]R_{12}(z)\frac{\D z}{2\pi {\rm i}}\\
\qquad=2\left( \int_{\delta {\rm e}^{\pi {\rm i}/3-{\rm i}0^{+}}}^{{\rm e}^{\pi {\rm i}/3-{\rm i}0^{+}}}-\int_{\delta {\rm e}^{\pi {\rm i}/3+{\rm i}0^{+}}}^{{\rm e}^{\pi {\rm i}/3+{\rm i}0^{+}}}
\right)\left[\Li_r\left( \frac{2z^{3}}{(1+z^3)^{2}} \right) \log \frac{2z^{3}}{(1+z^3)^{2}}\right]R_{12}(z)\frac{\D z}{2\pi {\rm i}}
\end{gather*}
holds for $ \delta\in(0,1)$. Thus, the branch cut of $ \log \frac{2z^{3}}{(1+z^3)^{2}}$ eventually forms a term in the $ \mathbb Q$-vector space $ \mathfrak Z_{r+1}(12)+\sqrt{3}\mathfrak Z_{r+1}(12)$ after contour integration, according to GPL recursion and~\eqref{eq:Zk(N)_defn''}.
In all, the difference between the right-hand sides of \eqref{eq:3kHkZr}
and
\eqref{eq:3kH3kZr} is ascribed to the participation of yet another GPL shuffling
\eqref{eq:ShGPL}, which increases the weight from $r$ to $r+1$.

The identity \eqref{eq:3kH3kZ1} is a special case of the procedures to be expounded in
Section~\ref{subsec:logSun}.

What remains to be checked is \eqref{eq:3kH2kZr}. In place of \eqref{eq:2kHkZr_L}, we now have
\begin{gather}
\sum_{k=0}^\infty\frac{2^k\binom{3k}k}{3^{3k}}\mathsf H_{2k}^{(r)}\nonumber\\
\qquad{}= \left(\int_{\infty {\rm e}^{2\pi {\rm i}/3}}^0+\int_0^\infty\right)
 \left[\Li_r\left( \frac{\sqrt{2}z^{3}}{1+z^6} \right) R_{24}^{+}(z)+\Li_r\left(- \frac{\sqrt{2}z^{3}}{1+z^6} \right) R_{24}^{-}(z)\right]\frac{\D z}{2\pi {\rm i}},\tag{\ref{eq:3kH2kZr}$^*$}\label{eq:3kH2kZr_star}
\end{gather}
where the rational functions\begin{align*}
R_{24}^\pm(z):= \frac{3\sqrt{3}{\rm i}}{2}\left(\frac{1}{ 1-{\rm e}^{4\pi {\rm i}/3}}+\frac{z^{2}}{1-{\rm e}^{2\pi {\rm i}/3}}\right)\frac{1}{1\mp\frac{\sqrt{2}z^{3}}{1+z^6}}\frac{z}{1+z^{3}}
\end{align*}have only simple poles at certain $ 24$th roots of unity. Furthermore, the residues at these poles all belong to $ \mathbb Q+\mathbb Q\sqrt 3$. Subsequent branch cut analyses and GPL fibrations of \eqref{eq:3kH2kZr_star} then complete the proof of \eqref{eq:3kH2kZr}.
\end{proof}

We can now transfer the procedures above to the analysis of Sun's series involving $ \binom{4k}{2k}$.

\begin{proof}[Proof of Theorem \ref{thm:3k4k}(c) for $\boldsymbol{\mathfrak Z_k(16)}$ and $ \boldsymbol{\mathfrak Z_{k}(24)}$,
where $\boldsymbol{k\in\mathbb Z_{>1}}$]
For the verification of equation~\eqref{eq:4kZr16}, consider \begin{align}
\sum_{k=0}^\infty\frac{\binom{4k}{2k}}{2^{5k}}\mathsf H_k^{(r)}=\left(\int_{{\rm i}\infty}^0+\int_0^\infty\right)\Li_r\left( \frac{2z^{4}}{(1+z^4)^{2}} \right) R_{16}(z)\frac{\D z}{2\pi {\rm i}},\tag{\ref{eq:4kZr16}$^*$}\label{eq:4kZr16_L}
\end{align} where the rational function\begin{align*}
R_{16}(z):= 2\sqrt{2}{\rm i}\left( \frac{1}{1-{\rm i}}+\frac{z^2}{1+{\rm i}} \right)\frac{1}{1-\frac{2z^{4}}{(1+z^4)^{2}}}\frac{1}{1+z^{4}}
\end{align*}has only simple poles at certain $16$th roots of unity, with residues in the $ \mathbb Q$-vector space\begin{align*}
\mathbb Q\sqrt{1+\frac{1}{\sqrt{2}}}+
\mathbb Q\sqrt{1-\frac{1}{\sqrt{2}}}.
\end{align*}One can thus deform the contour in \eqref{eq:4kZr16_L} to a tight loop $ C_{\varepsilon}^{[\pi/8,3\pi/8]}$ wrapping around the branch cut of the polylogarithm, as done in the proof of Theorem \ref{thm:3k4k}\,(b).

The rationale behind \eqref{eq:4kZr24} is similar.\end{proof}
\subsection{Logarithmic forms of Sun's series\label{subsec:logSun}}In \cite[Section~3]{Zhou2022mkMpl}, not only have we specified critical parameters in infinite series that produce CMZV evaluations, we have also investigated the ``inverse binomial sums'' and ``binomial sums'' involving $ \binom{2k}k$ and generic parameters, such as those on the right-hand sides of \eqref{eq:4k2k_a}--\eqref{eq:4k2k_b}.

In this subsection, we examine some convergent series involving $ \binom{3k}k$ and $ \binom{4k}{2k}$, in the general setting of Theorem \ref{thm:logSun}.

\begin{proof}[Proof of Theorem \ref{thm:logSun}\,(a)]According to our experience in Section~\ref{subsec:CMZVs}, we have the following companions to \eqref{eq:3kHklog} for $ \big|1+\varXi^3\big|^2<4\big|\varXi^3\big|$:
\begin{gather}
\frac{3\sqrt{3}}{2\pi}\left(\int_{\infty {\rm e}^{2\pi {\rm i}/3}}^0+\int_0^\infty\right)\frac{\Li_1\left( \frac{z^{3}(1+\varXi^3)^{2}}{\varXi^3(1+z^3)^{2}} \right)}{1-\frac{z^{3}(1+\varXi^3)^{2}}{\varXi^3(1+z^3)^{2}}}\frac{\D z}{1+z^3}\nonumber\\
\qquad{} = \big(1-{\rm e}^{2\pi {\rm i}/3}\big)\sum_{k=0}^\infty\frac{\binom{3k}k}{3^{3k}}\left[ \frac{\big(1+\varXi^3\big)^2}{\varXi^3} \right]^k\mathsf H_k,\tag{\ref{eq:3kHklog}$ ^\circ$}\label{eq:3kHklog_0}\\
\frac{3\sqrt{3}}{2\pi}\left(\int_{\infty {\rm e}^{2\pi {\rm i}/3}}^0+\int_0^\infty\right)\frac{\Li_1\left( \frac{z^{3}(1+\varXi^3)^{2}}{\varXi^3(1+z^3)^{2}} \right)}{1-\frac{z^{3}(1+\varXi^3)^{2}}{\varXi^3(1+z^3)^{2}}}\frac{z\D z}{1+z^3}\nonumber\\
\qquad{} = \big(1-{\rm e}^{4\pi {\rm i}/3}\big)\sum_{k=0}^\infty\frac{\binom{3k}k}{3^{3k}}\left[ \frac{\big(1+\varXi^3\big)^2}{\varXi^3} \right]^k\mathsf H_k,\tag{\ref{eq:3kHklog}$ ^\dagger$}\label{eq:3kHklog_1}\\
\frac{3\sqrt{3}}{2\pi}\left(\int_{\infty {\rm e}^{2\pi {\rm i}/3}}^0+\int_0^\infty\right)\frac{\Li_1\left( \frac{z^{3}(1+\varXi^3)^{2}}{\varXi^3(1+z^3)^{2}} \right)}{1-\frac{z^{3}(1+\varXi^3)^{2}}{\varXi^3(1+z^3)^{2}}}\frac{z^{2}\D z}{1+z^3}= 0.\tag{\ref{eq:3kHklog}$ ^\ddagger$}\label{eq:3kHklog_2}
\end{gather}
If we also know that $ \arg \varXi\in(0,2\pi/3)$, then each individual integrand here has two simple poles \big(namely $ \varXi$ and $ {\rm e}^{2\pi {\rm i}/3}/\varXi$\big) in the sector $ \arg z\in(0,2\pi/3)$. We can take a certain $ \mathbb C$-linear combination of the last three displayed equations, so that the resulting integrand is free from simple poles within the same sector:
\begin{gather}
\left(\int_{\infty {\rm e}^{2\pi {\rm i}/3}}^0+\int_0^\infty\right)\frac{\Li_1\left( \frac{z^{3}(1+\varXi^3)^{2}}{\varXi^3(1+z^3)^{2}} \right)}{1-\frac{z^{3}(1+\varXi^3)^{2}}{\varXi^3(1+z^3)^{2}}}\frac{3\sqrt{3}{\rm i}(z-\varXi)\big(z-{\rm e}^{2\pi {\rm i}/3}/\varXi\big)}{1+z^3}\frac{\D z}{2\pi {\rm i}}\nonumber\\
\qquad=\left[ \frac{3}{2} \left(\frac{1}{\varXi }-\varXi \right)-\frac{{\rm i} \sqrt{3}}{2} \left(\varXi +\frac{1}{\varXi }\right)+{\rm i} \sqrt{3} \right]\sum_{k=0}^\infty\frac{\binom{3k}k}{3^{3k}}\left[ \frac{\big(1+\varXi^3\big)^2}{\varXi^3} \right]^k\mathsf H_k.\tag{\ref{eq:3kHklog}$ ^\heartsuit$}\label{eq:3kHklog_h}
\end{gather}As we specialize \eqref{eq:3kHklog_h} to $ \varXi={\rm e}^{{\rm i}\varphi}$ for $ \varphi\in(0,\pi/3)$, and deform the contour of integration, we get\begin{align}
\sum_{k=0}^\infty\frac{\binom{3k}k}{3^{3k}}\left( 4 \cos^2\frac{3\varphi}{2}\right)^k\mathsf H_k=\lim_{\varepsilon\to0^+}\int_{C_\varepsilon^{[\varphi,2\pi/3-\varphi]}}\Li_1\left( \frac{4z^{3}\cos^2\frac{3\varphi}{2}}{(1+z^3)^{2}} \right)R_\varphi^{\vartriangle}(z)\frac{\D z}{2\pi {\rm i}},\tag{\ref{eq:3kHklog}$ ^\vartriangle$}\label{eq:3kHklog_tri}
\end{align}
where \begin{align*}
R_\varphi^{\vartriangle}(z):= {}&\frac{\sqrt{3}{\rm i}\sin \left(\frac{\varphi }{2}-\frac{2 \pi }{3}\right)}{\sin^{2}\frac{3 \varphi }{2}}\left[ \frac{{\rm e}^{\frac{{\rm i} \varphi }{2}-\frac{\pi {\rm i}}{3}} \sin \varphi }{z-{\rm e}^{{\rm i} \varphi +\frac{2 \pi {\rm i}}{3}}} -\frac{{\rm e}^{\frac{{\rm i} \varphi }{2}} \sin \left(\varphi +\frac{\pi }{3}\right)}{z-{\rm e}^{{\rm i} \varphi-\frac{2 \pi {\rm i}}{3}}}\right.\\{}&\left.{}-\frac{{\rm e}^{-\frac{{\rm i} \varphi }{2}+\frac{\pi {\rm i}}{3}} \sin \varphi}{z-{\rm e}^{-{\rm i} \varphi }}+\frac{{\rm e}^{-\frac{{\rm i} \varphi }{2}} \sin \left(\varphi +\frac{\pi }{3}\right)}{z-{\rm e}^{-{\rm i} \varphi-\frac{2 \pi {\rm i}}{3}}}\right].
\end{align*}In view of the branch cut behavior,
 \begin{align*}
\lim_{\varepsilon\to0^+}\left[\Li_1\left( \frac{4\big[(1+\varepsilon){\rm e}^{{\rm i}\phi}\big]^{3}\cos^2\frac{3\varphi}{2}}{\big\{1+\big[(1+\varepsilon){\rm e}^{{\rm i}\phi}\big]^{3}\big\}^{2}} \right)-\Li_1\left( \frac{4\big[(1-\varepsilon){\rm e}^{{\rm i}\phi}\big]^{3}\cos^2\frac{3\varphi}{2}}{\big\{1+\big[(1-\varepsilon){\rm e}^{{\rm i}\phi}\big]^{3}\big\}^{2}} \right)\right]=2\pi {\rm i}\frac{\phi-\frac{\pi}{3}}{\bigl\vert \phi-\frac{\pi}{3}\bigr\vert}
\end{align*}
for $ \phi\in[\varphi,\pi/3)\cup(\pi/3,2\pi/3-\varphi]$, we can compute the right-hand side of \eqref{eq:3kHklog_tri} by integrating four rational functions that are non-singular in a neighborhood of the loop $ C_\varepsilon^{[\pi,2\pi/3-\varphi]}$, so the final form of \eqref{eq:3kHklog} [given in Theorem \ref{thm:logSun}\,(a)] emerges after a little trigonometry.

To establish \eqref{eq:3kH3klog} in full, we examine\begin{align}
\frac{3\sqrt{3}}{2\pi}\left(\int_{\infty {\rm e}^{2\pi {\rm i}/3}}^0+\int_0^\infty\right)\frac{-3\log z-2\log\left( 1+\frac{1}{z^3} \right)}{1-\frac{4z^{3}\cos^2\frac{3\varphi}{2}}{(1+z^3)^{2}}}\frac{\D z}{1+z^3}\tag{\ref{eq:3kH3klog}$ ^\circ$}\label{eq:3kH3klog_o}
\end{align} in two ways. First, by tracking the integrand literally along the contour above, we see that~\eqref{eq:3kH3klog_o} is equal to
\begin{gather*}
\frac{3\sqrt{3}}{2\pi}\big(1-{\rm e}^{2\pi {\rm i}/3}\big)\int_0^\infty \frac{\log\frac{z^{3}}{(1+z^{3})^2}}{1-\frac{4z^{3}\cos^2\frac{3\varphi}{2}}{(1+z^3)^{2}}}\frac{\D z}{1+z^3}+2\pi {\rm i}{\rm e}^{2\pi {\rm i}/3}\int_0^\infty \frac{1}{1-\frac{4z^{3}\cos^2\frac{3\varphi}{2}}{(1+z^3)^{2}}}\frac{\D z}{1+z^3}\\
\qquad{}=\big(1-{\rm e}^{2\pi {\rm i}/3}\big)\sum_{k=0}^\infty\frac{\binom{3k}k}{3^{3k}}\left( 4\cos^2\frac{3\varphi}{2} \right)^k(
3\mathsf H_{3 k}-2\mathsf H_{2 k}-\mathsf H_k-3 \log 3)\\
\qquad\quad{} +2\pi {\rm i}{\rm e}^{2\pi {\rm i}/3}\sum_{k=0}^\infty\frac{\binom{3k}k}{3^{3k}}\left( 4\cos^2\frac{3\varphi}{2} \right)^k,
\end{gather*}
where the last infinite sum evaluates to $ \sin\left( \frac{\varphi}{2}+ \frac{\pi}{3}\right)\csc\frac{3\varphi}{2}$ \cite[Remark 2.7]{Sun2022}. Second, we close the contour in the sector $ \arg z\in[0,2\pi/3]$, taking care of the branch cut that is a straight line segment joining $0 $ to ${\rm e}^{\pi {\rm i}/3}$, while picking up residues at $ z={\rm e}^{{\rm i}\varphi}$ and $ z={\rm e}^{2\pi {\rm i}/3-{\rm i}\varphi}$ in the meantime. The net result then translates into \eqref{eq:3kH3klog} as stated in Theorem \ref{thm:logSun}\,(a).

 The proof of \eqref{eq:3kH2klog} is similar to that of \eqref{eq:3kHklog}. \end{proof}

\begin{proof}[Proof of Theorem \ref{thm:logSun}\,(b)]We have \begin{align*}\begin{split}
\sum_{k=0}^\infty\frac{\binom{4k}{2k}}{2^{6k}}\big(4\cos^22\psi\big)^k\mathsf H_k={}&\frac{2\sqrt{2}}{\pi}\frac{1}{1-{\rm i}}\left(\int_{{\rm i}\infty}^0+\int_0^\infty\right)\frac{\Li_1\left( \frac{4z^{4}\cos^22\psi}{(1+z^4)^{2}} \right)}{1-\frac{4z^{4}\cos^22\psi}{(1+z^4)^{2}}}\frac{\D z}{1+z^4}\\={}&\frac{2\sqrt{2}}{\pi}\frac{1}{1+{\rm i}}\left(\int_{{\rm i}\infty}^0+\int_0^\infty\right)\frac{\Li_1\left( \frac{4z^{4}\cos^22\psi}{(1+z^4)^{2}} \right)}{1-\frac{4z^{4}\cos^22\psi}{(1+z^4)^{2}}}\frac{z^{2}\D z}{1+z^4}\end{split}
\end{align*}by analogy to \eqref{eq:3kHklog_0} and \eqref{eq:3kHklog_1}. Instead of the zero-padding identity in \eqref{eq:3kHklog_2}, we now have a trailing term\begin{gather*}
\frac{2\sqrt{2}}{\pi}\left(\int_{{\rm i}\infty}^0+\int_0^\infty\right)\frac{\Li_1\left( \frac{4z^{4}\cos^22\psi}{(1+z^4)^{2}} \right)}{1-\frac{4z^{4}\cos^22\psi}{(1+z^4)^{2}}}\frac{z\D z}{1+z^4}=\frac{2\sqrt{2}}{\pi}\int_0^\infty \frac{\Li_1\left( \frac{4x^{2}\cos^22\psi}{(1+x^2)^{2}} \right)}{1-\frac{4x^{2}\cos^22\psi}{(1+x^2)^{2}}}\frac{\D x}{1+x^2}\\
\qquad=\sqrt{2}\sum_{k=0}^\infty\frac{\binom{2k}{k}}{2^{4k}}\big(4\cos^22\psi\big)^k\mathsf H_k=\frac{2\sqrt{2}}{\sin2\psi}\log \frac{1+\sin2 \psi}{2\sin 2 \psi }.
\end{gather*} Here, in the penultimate step, we have produced binomial coefficients from an integral representation of Euler's beta function \cite[item 8.380.3]{GradshteynRyzhik}; in the last step, we have quoted Boyadzhiev's formula \cite[formula (1.1)]{Sun2022}.

According to the information in the last passage, we can construct a relevant analog of \eqref{eq:3kHklog_tri} in the following form:\begin{gather}
\sum_{k=0}^\infty\frac{\binom{4k}{2k}}{2^{6k}}\big(4\cos^22\psi\big)^k\mathsf H_k-\frac{2\sin\big(\psi+\frac{\pi}{4}\big)}{\sin2\psi}\log \frac{1+\sin2 \psi}{2\sin 2 \psi }\nonumber\\
\qquad=\lim_{\varepsilon\to0^+}\int_{C_\varepsilon^{[\psi,\pi/2-\psi]}}\Li_1\left( \frac{4z^{4}\cos^22\psi}{(1+z^4)^{2}} \right)R_\psi^{\square}(z)\frac{\D z}{2\pi {\rm i}},\tag{\ref{eq:4kHklog}$ ^\square$}\label{eq:4kHklog_sq}
\end{gather}where \begin{align*}
R_{\psi}^\square(z):= \frac{\sqrt{2}(1+{\rm i})(z-{\rm e}^{{\rm i}\psi})(z-{\rm i}{\rm e}^{-{\rm i}\psi})}{1-\frac{4z^{4}\cos^22\psi}{(1+z^4)^{2}}}\frac{1}{1+z^4}.
\end{align*} We leave detailed calculations and subsequent arguments behind \eqref{eq:4kHklog_sq} to diligent readers.\end{proof}

Until now, our working examples for CMZVs of levels $N\in\{4,8,12,16,24\} $ are restricted to constructible cyclotomy, where all the algebraic numbers of our interest can be represented through (nested) square roots. The proofs in the current subsection are by no means limited to regular polygons constructible through ruler and compass.

To illustrate the generality of our methods, we point out that the proofs of Theorems \ref{thm:3k4k}\,(b) and \ref{thm:logSun}\,(a) can be readily adapted to the following statements for all $ r\in\mathbb Z_{>0}$:
\begin{gather}
\sum_{k=0}^\infty\frac{\binom{3k}k}{3^{3k}}\mathsf H_k^{(r)}=\frac{3\sqrt{3}}{2\pi}\int_0^\infty\frac{\Li_r\bigl( \frac{X^{3}{}}{(1+X^3)^{2}} \bigr)}{1-\frac{X^{3}{}}{(1+X^3)^{2}}}\frac{\D X}{1+X^3}\in  c_{2/9}\mathfrak Z_r(18)+c_{4/9} \mathfrak Z_r(18),\label{eq:Zr18}\\\sum_{k=0}^\infty\frac{\binom{3k}k}{3^{2k}}\mathsf H_k^{(r)}=\frac{3\sqrt{3}}{2\pi}\int_0^\infty\frac{\Li_r\bigl( \frac{3X^{3}{}}{(1+X^3)^{2}} \bigr)}{1-\frac{3X^{3}{}}{(1+X^3)^{2}}}\frac{\D X}{1+X^3}\in  c_{2/9}\mathfrak Z_r(9)+c_{4/9}\mathfrak Z_r(9),\label{eq:Zr9}
\end{gather}
where $ c_\nu:= \cos(\nu\pi)$. Here, if $ r\in\{1,2\}$, then we can sharpen the right-hand side of \eqref{eq:Zr18} into $ c_{2/9}\mathfrak Z_r(9)+c_{4/9}\mathfrak Z_r(9)$ by Au's automated reduction \cite{Au2022a} of the corresponding GPL output.

Defining further $ s_\nu:= \sin(\nu\pi)$ and $ \lambdabar_\nu:= \log(2s_\nu)$, we have the following instances of \eqref{eq:Zr18} and~\eqref{eq:Zr9}: \begin{gather*}
\sum_{k=0}^\infty\frac{\binom{3k}k}{3^{3k}}\mathsf H_k=-2(3c_{2/9}+2c_{4/9})\lambdabar_{1/9}-2(2c_{2/9}+5c_{4/9})\lambdabar_{2/9}+2(c_{4/9}-c_{2/9}\mathfrak )\lambdabar_{1/3},\\
\sum_{k=0}^\infty\frac{\binom{3k}k}{3^{2k}}\mathsf H_k=-6(c_{2/9}+c_{4/9})\lambdabar_{1/9}+6c_{2/9}\lambdabar_{2/9}-2(2c_{2/9}+c_{4/9})\lambdabar_{1/3},\\
\sum_{k=0}^\infty\frac{\binom{3k}k}{3^{3k}}\mathsf H_k^{(2)}=12 (c_{2/9}+3 c_{4/9}) \Li_2(2c_{4/9})-18 (c_{2/9}+2 c_{4/9}) \Li_2\big(4 c_{4/9}^2\big)\\
\hphantom{\sum_{k=0}^\infty\frac{\binom{3k}k}{3^{3k}}\mathsf H_k^{(2)}=}{}
-18 c_{4/9} \Li_2\left(\frac{4 }{3}s_{1/9}^2\right)-6 c_{4/9} \bigl(\lambdabar_{1/9}^2-2\lambdabar_{1/9} ^{}\lambdabar_{2/9}^{}-2 \lambdabar_{2/9}^2+\lambdabar_{1/3}^2 \bigr)\\
\hphantom{\sum_{k=0}^\infty\frac{\binom{3k}k}{3^{3k}}\mathsf H_k^{(2)}=}{}
-12 c_{4/9} (2 \lambdabar_{1/9}+\lambdabar_{2/9} ) \lambdabar_{1/3}-\frac{2(c_{2/9}+3 c_{4/9})\pi^2}{9},\\
\sum_{k=0}^\infty\frac{\binom{3k}k}{3^{2k}}\mathsf H_k^{(2)}=-12 (c_{2/9}- c_{4/9}) \Li_2(2c_{4/9})+12 (c_{2/9}- c_{4/9}) \Li_2\big(4 c_{4/9}^2\big)\\
\hphantom{\sum_{k=0}^\infty\frac{\binom{3k}k}{3^{2k}}\mathsf H_k^{(2)}=}{}
+12 c_{2/9}\Li_2\left(\frac{4 }{3}s_{1/9}^2\right)+6c_{2/9} ( \lambdabar_{1/9}+\lambdabar_{1/3} )^2+\frac{2(c_{2/9}- c_{4/9})\pi^2}{9},
\end{gather*}
and the list goes on.

\subsection*{Acknowledgements}

This research was supported in part by the Applied Mathematics Program within the Department of Energy
(DOE) Office of Advanced Scientific Computing Research (ASCR) as part of~the Collaboratory on
Mathematics for Mesoscopic Modeling of Materials (CM4).

I am truly grateful to Erik Panzer and Kam Cheong Au, whose software packages \texttt{HyperInt} and \texttt{MultipleZetaValues} furnished me with many concrete computational examples that inspired the present theoretical framework. My thanks are due to Zhi-Wei Sun for his thoughtful feedback on the initial draft of the manuscript, as well as the anonymous referees for their perceptive and constructive comments.

\pdfbookmark[1]{References}{ref}
\LastPageEnding

\end{document}